\documentclass[11pt,reqno]{amsart}
\usepackage{amsmath,amsthm,amsfonts,amssymb,mathrsfs,bm,graphicx,stmaryrd}

\usepackage{mathtools}
\usepackage{dsfont}
\usepackage{multicol}
\usepackage[colorlinks=true,linkcolor=blue]{hyperref}
\hypersetup{bookmarksdepth=1}
\usepackage{enumitem}
\usepackage{bbm}
\usepackage{mathrsfs}
\usepackage{versions}
\usepackage{subcaption}

\usepackage[letterpaper,hmargin=1.0in,vmargin=1.0in]{geometry}
\parskip 	\smallskipamount

\newcommand{\E}{\mathbb{E}}
\renewcommand{\P}{\mathbb{P}}

\newcommand{\cd}{\mathcal{D}}

\newcommand{\Z}{\mathbb{Z}}

\newcommand{\R}{\mathbb{R}}

\newcommand{\cl}{\mathcal{L}}
\newcommand*{\bo}{\boldsymbol}
\newcommand*{\wt}{\widetilde}

\newcommand{\ca}{\mathcal{A}}
\newcommand{\cb}{\mathcal{B}}
\newcommand{\cc}{\mathcal{C}}

\newcommand{\vep}{\varepsilon}

\newtheorem{theorem}{Theorem}[section]

\newtheorem{lemma}[theorem]{Lemma}
\newtheorem{proposition}[theorem]{Proposition}

\theoremstyle{remark}
\newtheorem{remark}{\bf Remark}

\numberwithin{equation}{section}

\begin{document}
\title{Limit theorems for extrema of Airy processes}
\author[Basu and Bhattacharjee]{Riddhipratim Basu and Sudeshna Bhattacharjee}

\begin{abstract}
    We establish limit theorems for the maxima and minima of Airy$_1$ and Airy$_2$ processes (denoted by $\ca_1(\cdot)$ and $\ca_2(\cdot)$ respectively) over growing intervals. In particular, we identify the finite non-zero constants that are the almost sure limits of $(\log t)^{-2/3}\max_{0\le s \le t} \ca_{i}(s)$ and $(\log t)^{-1/3}\min_{0\le s \le t} \ca_{i}(s)$ for $i=1,2$. This complements and extends the results of \cite{P23}, where the question for the maxima was considered and the order of growth was identified for both $\ca_1$ and $\ca_2$ and the constant was identified for $\ca_1$. Our approach is different from that of \cite{P23}; instead of complicated formulae for multi-point distributions, we rely on the well-known convergence of passage time profiles in planar exponential last passage percolation started from different initial conditions to $\ca_1$ and $\ca_2$, together with the recently developed sharp one-point estimates \cite{BBBK24} for the point-to-point and point-to-line passage times in exponential LPP and a combination of old and new results on the geometry of the LPP landscape.
\end{abstract}
\address{Riddhipratim Basu, International Centre for Theoretical Sciences, Tata Institute of Fundamental Research, Bangalore, India} 
\email{rbasu@icts.res.in}
\address{Sudeshna Bhattacharjee, Department of Mathematics, Indian Institute of Science, Bangalore, India}
\email{sudeshnab@iisc.ac.in }

\maketitle

\section{Introduction and Statement of Main results}
Airy$_1$ and Airy$_2$ processes (denoted $\ca_1(\cdot)$ and $\ca_2(\cdot)$ respectively) are stationary stochastic processes on $\R$ that arise in various contexts such as random matrix theory, random growth models, non-intersecting Brownian motions among others. One-point marginals of these processes are scalar multiples of GOE \cite{S05, BFPS07} and GUE Tracy-Widom \cite{PSpng02} distribution respectively. Various properties such as decay of correlations \cite{BBF22,AM03,AM05,W04,PSpng02} and ergodicity \cite{P23,PSpng02, CS14} have also been established using a variety of analytic and probabilistic techniques. {In a recent work \cite{P23}, the question of growth of the maxima of Airy processes over growing intervals was considered and a result about maxima of $\text{Airy}_1$ process and a partial result about maxima of $\text{Airy}_2$ process was proved. In this paper, we complete this picture by establishing law of fractional logarithms for \emph{both maxima and minima} of Airy$_1$ and Airy$_2$ processes. In particular, we answer a question regarding the maxima of Airy$_2$ process left open in \cite[Theorem 1.5]{P23}, and also provide a different proof of the limit theorem established for the maxima of Airy$_1$ in \cite[Theorem 1.4]{P23}, that avoids the complicated analysis of exact formulae. 

The following theorem on the growth of extrema of Airy$_2$ process is our first main result.
\begin{theorem}
\label{t:airy2limit}
For the Airy$_2$ process $\ca_2(\cdot)$ we have 
\begin{enumerate}[label=(\roman*), font=\normalfont]
\item $$\lim_{t\to \infty}  \frac{ \max_{0\le s \le t}\ca_2(s)}{(\log t)^{2/3}}=\left(\frac{3}{4}\right)^{2/3}\text{ almost surely}. $$ 
\item $$\lim_{t\to \infty}  \frac{ \min_{0\le s \le t}\ca_2(s)}{(\log t)^{1/3}}=-(12)^{1/3} \text{ almost surely}.$$
\end{enumerate}
\end{theorem}
Partial progress towards Theorem \ref{t:airy2limit}, (i) was made in \cite[Theorem 1.5]{P23}, where the upper bound was proved relying on an estimate from \cite{CHH23}. However, the polynomial decay of correlations in $\ca_2$ was not sufficient to derive a matching lower bound; instead a lower bound of $(\frac{1}{4})^{2/3}$ was proved in \cite[Theorem 1.5]{P23}. We shall again provide a (slightly) different argument for the upper bound as well relying on one point estimates and some geometric considerations.

Our second main result is a similar theorem for the Airy$_1$ process.

\begin{theorem}
    \label{t:airy1limit}
    For the Airy$_1$ process $\ca_1(\cdot)$ we have 
     \begin{enumerate}[label=(\roman*), font=\normalfont]
\item $$\lim_{t\to \infty}  \frac{ \max_{0\le s \le t}\ca_1(s)}{(\log t)^{2/3}}=\left(\frac{3}{4\sqrt{2}}\right)^{2/3} \text{ almost surely}.$$
\item $$\lim_{t\to \infty}  \frac{ \min_{0\le s \le t}\ca_1(s)}{(\log t)^{1/3}}=-3^{1/3} \text{ almost surely}. $$
\end{enumerate}
\end{theorem}

As mentioned before, Theorem \ref{t:airy1limit}, (i) was already proved in \cite[Theorem 1.4]{P23}. The lower bound was proved using the superexponential decay of correlations for $\ca_1$ \cite{BBF22}, while the upper bound relied on a complicated Fredholm determinant formula for the distribution of the maximum of $\ca_1$ on a small interval from \cite{QR13}. Our proof will be different, and will avoid the analysis of complicated exact formulae; we shall rely on one point estimates and geometric arguments.  
\subsection{Airy processes and last passage percolation}
 We shall only rely on the fact that both $\ca_1$ and $\ca_2$ are obtained as weak limits of certain exponential last passage percolation (LPP) models
 (we shall also rely on the ergodicity of $\ca_1$ and $\ca_2$ but not in a very crucial way); all our probabilistic estimates are done at the level of the LPP model. Let us now state the relevant definitions of the relevant exponential LPP models.

Consider last passage percolation on $\Z^2$ with i.i.d.\ $\mbox{Exp}(1)$ passage times on the vertices. For $s\in \R$, let $u=u_{N}(s)=(N-\lfloor s(2N)^{2/3} \rfloor, N+\lfloor s(2N)^{2/3} \rfloor)\in \Z^2$. Let $T_{N}(s):=T_{\mathbf{0},u}$ denote the last passage time from $(0,0)$ to $u$ (i.e., the maximal total weight of an up/right path connecting $(0,0)$ and $u$, excluding the last vertex). We have the following theorem \cite[Theorem 3.8]{BGZ21}. 

\begin{theorem}
    \label{t:lppairy2limit}
    As $N\to \infty$, 
    \begin{equation}
    \label{e:airy2limit}
    \frac{T_{N}(s)-4N}{2^{4/3}N^{1/3}}\Rightarrow \ca_2(s)-s^2
\end{equation}
where $\Rightarrow$ denotes weak convergence in the topology of uniform convergence on compact sets. 
\end{theorem}

For $u=u_{N}(s)$ as before let $T^*_{N}(s)$ denote the last passage time from $\cl_0$ (the line $x+y=0$) to $u$ (i.e., the maximal weight among all paths that start at some point in $\cl_0$ and end at $u$ excluding the last vertex). Then we have the following theorem. 

\begin{theorem}
    \label{t:lppairy1limit}
    As $N\to \infty$, 
    \begin{equation} 
\label{eq: airy_1 convergence}
\frac{T^*_{N}(s)-4N}{2^{4/3}N^{1/3}}\Rightarrow 2^{1/3}\ca_1(2^{-2/3}s)
\end{equation}
where, as before, $\Rightarrow$ denotes weak convergence in the topology of uniform convergence on compact sets.
\end{theorem}

For both Theorems \ref{t:lppairy2limit} (see \cite{BF08}) and \ref{t:lppairy1limit} (see \cite{BFPS07,S05,CPF12}), finite dimensional convergence are known by results in integrable probability. Upgrading this to weak convergence by proving tightness was done in the case of Airy$_2$ in \cite{BGZ21}, but we could not find anywhere in the literature where Theorem \ref{t:lppairy1limit} is explicitly stated and proved. For the sake of completeness,  we shall provide a similar argument proving Theorem \ref{t:lppairy1limit}.

\subsection{One point estimates in LPP}
We record the one point upper and lower tail moderate deviation estimates here; these will be used throughout the paper. We first write the upper tail bound for the point-to-line LPP; note that the distribution of $T^*_{N}(s)$ is independent of $s$.

\begin{proposition}{\cite[Theorem 1.6, (i)]{BBBK24}}
    \label{l:p2lupper}
    Fix $\varepsilon>0$. Then for every $t$ sufficiently large (depending on $\vep$) and for $N$ sufficiently large (depending on $\vep, t$) we have 
    $$\exp \biggl(-\bigl(\frac{4}{3}+\varepsilon\bigr)t^{3/2}\biggr)\le \P\left(T^*_{N}(0)\ge 4N+t2^{4/3}N^{1/3}\right)\le \exp \biggl(-\bigl(\frac{4}{3}-\varepsilon\bigr)t^{3/2}\biggr).$$
\end{proposition}

  Notice that the same result is also present in \cite[Lemma A.4]{BBF22} 
 where the upper bound is stated but the proof actually gives both the bounds. These bounds are proved using a random matrix argument in \cite{BBBK24} for Laguerre $\beta$-ensembles (point-to-line and point-to-point LPP correspond to special cases $\beta=1$ and $\beta=2$ respectively).

  Next we record the upper tail bound for point-to-point LPP. 

\begin{proposition}\cite[Theorem 1.1, (ii), Theorem 1.3, (ii)]{BBBK24}
    \label{l:p2pupper}
    Fix $\varepsilon>0, s_0>0$. Then for every $t$ sufficiently large (depending on $\vep$) and $s\in [0,s_0]$ and for $N$ sufficiently large (depending on $s_0,t, \vep$) we have 
    $$\exp \biggl(-\bigl(\frac{4}{3}+\varepsilon\bigr)t^{3/2}\biggr)\le \P\left(T_{N}(s)\ge 4N-s^22^{4/3}N^{1/3}+t2^{4/3}N^{1/3}\right)\le \exp \biggl(-\bigl(\frac{4}{3}-\varepsilon\bigr)t^{3/2}\biggr).$$
\end{proposition}

In  \cite{BBBK24}, the above proposition is also proved using a random matrix argument, but one can also prove these result from an analysis of a Fredholm determinant formula as in \cite[Proposition A.4]{BBF22}.

Studies of upper tails in integrable last passage percolation, are by now rather standard, it involves standard analysis of exact formulae. However, while considering the lower tail, these formulae involve oscillating integrals and therefore are much harder to analyse. To derive optimal estimates for the lower tail, an approach based on Riemann-Hilbert methods have been successfully undertaken for Poissonian LPP \cite{LMS02} and Geometric LPP \cite{BDMMZ02} but such results do not appear to be available in the exponential LPP set-up. For the lower tail estimates we quote \cite{BBBK24} which again uses random matrix methods to prove such tail estimates.

\begin{proposition}\cite[Theorem 1.6, (ii)]{BBBK24}
    \label{l:p2llower}
    Fix $\varepsilon>0$. Then for every $t$ sufficiently large (depending on $\vep$) and for $N$ sufficiently large (depending on $t, \vep$) we have 
    $$\exp \biggl(-\bigl(\frac{1}{6}+\varepsilon\bigr)t^{3}\biggr)\le \P\left(T^*_{N}(0)\le 4N-t2^{4/3}N^{1/3}\right)\le \exp \biggl(-\bigl(\frac{1}{6}-\varepsilon\bigr)t^{3}\biggr).$$
\end{proposition}

\begin{proposition}
    \label{l:p2plower}\cite[Theorem 1.4, (ii), Theorem 1.5, (ii)]{BBBK24}
    Fix $\varepsilon>0, s_0>0$. Then for every $t$ sufficiently large (depending on $\vep$), $s\in [0,s_0]$ and for $N$ sufficiently large (depending on $t,s_0, \vep$) we have $$ \exp \biggl(-\bigl(\frac{1}{12}+\varepsilon\bigr)t^{3}\biggr) \leq \P\left(T_{N}(s)\le 4N-s^2 2^{4/3}N^{1/3}-t2^{4/3}N^{1/3}\right)\le \exp \biggl(-\bigl(\frac{1}{12}-\varepsilon\bigr)t^{3}\biggr).$$
\end{proposition}

\subsection{Our contributions and some words about the proofs}
The main technical contributions of this paper are eight estimates regarding last passage time profiles with droplet (Theorems \ref{t:2maxub}, \ref{t:2maxlb}, \ref{thm:minima of airy_2 upper bound}, \ref{t:2minlb}) and flat (Theorems \ref{thm: point to line max limit upper bound}, \ref{thm: point to line max limit lower bound}, \ref{thm: point to line min limit upper bound}, \ref{thm: point to line min limit lower bound}) initial conditions. These estimates correspond to upper and lower bounds in items (i) and (ii) in Theorems \ref{t:airy2limit} and Theorem \ref{t:airy1limit} respectively which are obtained from the LPP estimates by appealing to the weak convergence results in Theorems \ref{t:lppairy2limit} and Theorem \ref{t:lppairy1limit}. Out of these eight estimates, the limiting versions of three (upper bounds for the maxima of Airy$_1$ and Airy$_2$, and the lower bound for maxima of Airy$_1$) were previously proved in \cite[Theorem 1.4, Theorem 1.5]{P23} using very different kind of methods (see below); but we believe that the pre-limiting estimates at the LPP level are of independent interest. The other five estimates, as far as we know, are new even in the Airy limit. There are two different kinds of arguments that go into the proofs of these eight estimates which we briefly outline below. 

For four estimates involving probability upper bounds (i.e., to show that the maximums cannot to be too large or the minimums cannot be too small), we essentially show, using sharp one point estimates (Propositions \ref{l:p2lupper}, \ref{l:p2llower}, \ref{l:p2pupper}, \ref{l:p2plower}) that the probability of having a too large (or correspondingly a too small) passage time value in an interval of scaled length $O(1)$ (i.e., of unscaled length $N^{2/3}$), is $o(t^{-1})$ and the corresponding results follow by taking a union bound. This, of course, is the most natural approach for this problem, and essentially the same was employed in \cite{P23} for upper bounding the maximum of $\ca_1$ and $\ca_2$; utilising a result from \cite{CHH23} for the $\ca_2$ case and analysing a Fredholm determinant formula in the $\ca_1$ case. To obtain these four estimates on the maxima and minima of LPP profiles on intervals of scaled length 1 (Lemma \ref{lemma: itl rt ub}, Lemma \ref{lemma: min interval to line}, Lemma \ref{lemma: itp rt ub }, Lemma \ref{lemma: itp lt ub}) from the one-point tails bounds require several arguments, which are of independent interest. In fact, one of these estimates had already been obtained in a related context in \cite{BBF22}, but the other three are new. Once these estimates are obtained, the final step involving the union bound is essentially the same in all the four cases and hence we shall only provide the details the first time. 

The four other estimates involving probability lower bounds (i.e., showing that with positive probability the maxima is sufficiently large or the minima is sufficiently small) are generally much more complicated, and the LPP picture is crucially used. The idea is basically to show that the events in question can essentially be approximated by the union of $\approx t$ many independent events each of which has probability $\approx t^{-1}$. Using independence we then argue that the probability of at least one of these events hold is bounded away from 0. Designing these independent events require careful constructions crucially depending on the last passage percolation picture and on precise controls on the behaviour of the geodesics there. These estimates Lemma \ref{lemma: restriced passage time large}, Lemma \ref{l:restricted}, Lemma \ref{lemma: itp lt lb} are also of independent interest. Limiting version of one of these four estimates (corresponding to the lower bound of the maxima of $\ca_1$) was proved in \cite[Theorem 1.4]{P23} using the super-exponential decay of correlations in $\ca_1$, however one cannot do a similar argument for $\ca_2$ to get the correct lower bound since the correlations decay much more slowly there. 

Apart from the estimates described above, Theorem \ref{t:lppairy1limit} is also new; we could not find a reference with a statement of weak convergence of line-to-point LPP profile to $\ca_1$. The proof of this is fairly standard, but Theorem \ref{t:lppairy1limit} lets us extract some interesting properties of $\ca_1$ from LPP, e.g.\ strong mixing and positive association (see Remark \ref{r:sm} and Remark \ref{r:ass}); ergodicity and positive association was previously shown in \cite[Theorem 1.1, Theorem 1.2]{P23} using different arguments. 

\subsection{Notations:} Here we define some notations which we will use frequently throughout this paper. For any $r \in \Z, \bo{r}$ denotes $(r,r).$ For any $r \in \Z, \cl_{r}$ denotes the line $x+y=r.$ Sometimes we will work with the rotated axes $x+y=0$ and $x-y=0$. They will be called space axis and time axis respectively. For a vertex $v \in \mathbb{Z}^2$, $\phi(v)$ will denote the time coordinate of $v$ and $\psi(v)$ will denote the space coordinate of $v$. Precisely, we define
\begin{displaymath}
    \phi(u,v):=u+v,\qquad\psi(u,v):=u-v.
\end{displaymath}
For any vertices $u,v \in \Z^2$, such that $u \leq v$ (i.e., if $u=(u_1,u_2)$ and $v=(v_1,v_2),u_1 \leq u_2,v_1 \leq v_2$) $T_{u,v}$ will denote the last passage time between $u$ and $v$ (i.e., the maximal total weight of an up-right path between $u$ and $v$ excluding the last vertex). Similarly, for $u \in \Z^2$ and $r \in \Z$ such that $\phi(u) \geq r, T_{\cl_r,u}$ denotes the last passage time between $u$ and $\cl_r$ (i.e., the maximal total weight of an up-right path that starts at some point of $\cl_r$ and ends at $u$ excluding the last vertex). $\Gamma_{u,v}$ and $\Gamma_{\cl_r,u}$ will denote the a.s. unique geodesics (i.e., the maximal weight attaining path) in the above two cases respectively. For any such geodesic $\Gamma, \Gamma(r)$ will denote the random intersection point of $\Gamma$ and $\cl_r.$

\subsection{Organisation of the paper}
We prove Theorems \ref{t:airy1limit} and \ref{t:airy2limit} in Sections \ref{s:a1} and \ref{s:a2} respectively, assuming certain LPP estimates that are consequences of Propositions \ref{l:p2lupper}, \ref{l:p2pupper}, \ref{l:p2llower} and \ref{l:p2plower}. These estimates are proved in Section \ref{s:lpp}. Finally, Theorem \ref{t:lppairy1limit} is proved in Section \ref{s:weak}. 

\subsection*{Acknowledgements}RB is partially supported by a MATRICS grant (MTR/2021/000093) from SERB, Govt.\ of India, DAE project no.\ RTI4001 via ICTS, and the Infosys Foundation via the Infosys-Chandrasekharan Virtual Centre for Random Geometry of TIFR. SB is supported by scholarship from National Board for Higher Mathematics (NBHM) (ref no: 0203/13(32)/2021-R\&D-II/13158).

\section{Extrema of $\text{Airy}_1$ process: Proof of Theorem \ref{t:airy1limit}}
\label{s:a1}
We first prove Theorem \ref{t:airy1limit}, (i). As mentioned before, Theorem \ref{t:airy1limit}, (i) was already proved in \cite[Theorem 1.4]{P23}. We give a different proof here. Notice that using Theorem \ref{t:lppairy1limit}, and the ergodicity of the Airy$_1$ process \cite[Theorem 1.1]{P23}, Theorem \ref{t:airy1limit} follows immediately from the next two results.
\begin{theorem}
\label{thm: point to line max limit upper bound}For any $\vep>0$, there exists $\gamma>0$ such that for all sufficiently large $t$ (depending on $\vep$) and sufficiently large  $N$ (depending on $t, \vep$) 
\[
\P \left( \max_{0 \leq s \leq t}\frac{T_N^*(s)-4N}{2^{4/3}N^{1/3}(\log t)^{2/3}} \geq \left(\frac{3}{4}\right)^{2/3}(1+\vep) \right) \leq t^{-\gamma}.
\]
\end{theorem}

\begin{theorem}
\label{thm: point to line max limit lower bound}
    For any $\vep>0,$ for all $t$ sufficiently large (depending only on $\vep$) and for sufficiently large $N$ (depending only on $t, \vep$)
\[
\P \left( \max_{0 \leq s \leq t}\frac{T_N^*(s)-4N}{2^{4/3}N^{1/3}(\log t)^{2/3}} \geq \left(\frac{3}{4}\right)^{2/3}(1-\vep) \right) \geq \frac 12,
\]
\end{theorem}

\begin{proof}[Proof of Theorem \ref{t:airy1limit}, (i)]
{
First note that both 
$$\limsup_{t\to \infty} \dfrac{\max_{0\le s \le t} \ca_1(s)}{(\log t)^{2/3}}~~~~\text{and}~~~~ \liminf_{t\to \infty} \dfrac{\max_{0\le s \le t} \ca_1(s)}{(\log t)^{2/3}}$$ are invariant under translations. Indeed using the facts that for $r>0$, $\dfrac{\max_{0\le s \le r} \ca_1(s)}{(\log t)^{2/3}}\to 0$ and  $\frac{\log (t+r)}{\log t}\to 1$ as $t\to \infty$ it follows that almost surely 
\[
\limsup_{t \rightarrow \infty}\frac{\max_{0 \leq s \leq t}\ca_1(s+r)}{(\log t)^{2/3}}=\limsup_{t \rightarrow \infty}\frac{\max_{0 \leq s \leq t}\ca_1(s)}{(\log t)^{2/3}}.
\]
The other cases can be checked similarly and the details are omitted. 

The shift invariance, together with the ergodicity of $\ca_1$ \cite[Theorem 1.1]{P23} (see also Remark \ref{r:sm} for an independent proof), implies that  we have, for some constants $c_1,c_2$ (possibly infinite), almost surely
$$\limsup_{t\to \infty} \dfrac{\max_{0\le s \le t} \ca_1(s)}{(\log t)^{2/3}}=c_1, \quad \liminf_{t\to \infty} \dfrac{\max_{0\le s \le t} \ca_1(s)}{(\log t)^{2/3}}=c_2.$$

Next using the continuity of the function $f \mapsto \sup_{K} f$ for compact sets $K$ in the topology of uniform convergence on compacts it follows from Theorem \ref{t:lppairy1limit} that
$\max_{0 \leq s \leq t}\frac{T_N^*(s)-4N} {2^{4/3}N^{1/3}(\log t)^{2/3}}$ converges in distribution to $\max_{0 \leq s \leq t} \frac{\ca_1(s)}{(\log t)^{2/3}}$. Therefore, we have by Theorem \ref{thm: point to line max limit upper bound}, there exists $\gamma>0$ depending on $\vep$ such that
\begin{align*}
&\P \left(\max_{0 \leq s \leq t}\frac{\ca_1(s)}{(\log t)^{2/3}} >\left( \frac 34\right)^{2/3}(1+\vep) \right)\\
&\leq \liminf_{N \rightarrow \infty}\P \left(\max_{0 \leq s \leq t}\frac{T_N^*(s)-4N} {2^{4/3}N^{1/3}(\log t)^{2/3}} > \left(\frac 34 \right)^{2/3}(1+\vep)\right) \leq t^{-\gamma}.
\end{align*}
Similarly, by Theorem \ref{thm: point to line max limit lower bound}
\begin{align*}
    &\P \left(\max_{0 \leq s \leq t}\frac{\ca_1(s)}{(\log t)^{2/3}} \geq \left( \frac 34\right)^{2/3}(1-\vep) \right)\\ & \geq \limsup_{N \rightarrow \infty} \P \left(\max_{0 \leq s \leq t}\frac{T_N^*(s)-4N} {2^{4/3}N^{1/3}(\log t)^{2/3}} \geq \left(\frac 34 \right)^{2/3}(1-\vep)\right) \geq \frac 12.
\end{align*}} 

It follows that $c_1\le 2^{-1/3}(\frac{3}{4})^{2/3}= (\frac{3}{4\sqrt{2}})^{2/3}$ and $c_2\ge 2^{-1/3}(\frac{3}{4})^{2/3}= (\frac{3}{4\sqrt{2}})^{2/3}$. Combining these the proof is completed. 
\end{proof}

A couple of remarks are in order. In the proofs of Theorem \ref{t:airy1limit}, (ii) and Theorem \ref{t:airy2limit}, the use of ergodicity and weak convergence will be identical to those in the above proof, so we shall skip the details. We also point out that the use of ergodicity in these proofs are convenient but not essential. Indeed, in the upper bound above one can simply take the sequence $t_i=2^i$ and apply Borel-Cantelli Lemma (together with the fact that $\max_{0\le s\le t} \ca_1(s)$ is increasing in $t$ and $\log t_{i+1}/\log t_i\to 1$) to get the result. One can also check from the proof of Theorem \ref{thm: point to line max limit lower bound} that the lower bound $1/2$ there can be replaced by $1-1/t^2$ upon which a similar application of Borel-Cantelli Lemma completes the proof of the lower bound.

We now move towards proving Theorem \ref{thm: point to line max limit upper bound}. We will need the following lemma. Let $I$ denote the interval of length $\lfloor (2N)^{2/3}\rfloor$ on $\cl_{2N}$ with midpoint $\bo{N}.$
\begin{lemma}\cite[Proposition A.5]{BBF22}
\label{lemma: itl rt ub}
For any $N^{2/3} \gg t \geq 0$
\[
\P \left( \max_{x \in I}T_{\cl_{0},x} \geq 4N+2^{4/3}tN^{1/3}\right) \leq C \max \{1,t\}e^{-\frac 43 t^{3/2}},
\]
for some constant $C>0.$
    
\end{lemma}

\begin{proof}[Proof of Theorem \ref{thm: point to line max limit upper bound}]
Let $\vep>0$ be chosen arbitrarily. We wish to calculate
\[
\P \left( \max_{0\leq s \leq t}\frac{T_N^*(s)-4N}{2^{4/3}N^{1/3}(\log t)^{2/3}} \geq \left(\frac 34 \right)^{2/3}(1+\vep)\right).
\]
We apply a union bound. We chose $\vep'$ so that 
\[
  (1+\vep)^{3/2}(1-\vep') >1.
\]
Consider the interval $J^t,$ on $\cl_{2N}$ with end points $\bo{N}$ and $(N-\lfloor t(2N)^{2/3} \rfloor,N+\lfloor t(2N)^{2/3} \rfloor)$. We divide it into sub-intervals each with length $\lfloor (2N)^{2/3} \rfloor.$ Without loss of generality assume that $t$ is an integer and there are $t$ many such intervals.
For $1 \leq i \leq t$, we denote these intervals by $J_i$. We have
\[
\Big \{ \max_{0\leq s \leq t}\frac{T_N^*(s)-4N}{2^{4/3}N^{1/3}(\log t)^{2/3}} \geq \left(\frac 34 \right)^{2/3}(1+\vep) \Big \} \subset \bigcup_{i=1}^{t} \Bigg \{ \max_{x \in J_i} \frac{T_{\cl_{0},x}-4N}{2^{4/3}N^{1/3}( \log t)^{2/3}} \geq \left( \frac 34 \right)^{2/3}(1+\vep) \Bigg\}.
\]
Now, by Lemma \ref{lemma: itl rt ub} we have for sufficiently large $N$ and $t$ with $N^{2/3} \gg t \geq 0$ and $t$ sufficiently large, for each fixed $i$
\[
\P \left( \max_{x \in J_i } \frac{T_{\cl_{0},x}-4N}{2^{4/3}N^{1/3}( \log t)^{2/3}} \geq \left( \frac 34 \right)^{2/3}(1+\vep)\right) \leq e^{-(1+\vep)^{3/2}(1-\vep') \log t }.
\]
Hence, we have 
\[
\P \left( \max_{0\leq s \leq t}\frac{T_N^*(s)-4N}{2^{4/3}N^{1/3}(\log t)^{2/3}} \geq \left(\frac 34 \right)^{2/3}(1+\vep)\right) \leq t e^{-(1+\vep)^{3/2}(1-\vep') \log t }.
\]
Now, by choice of $\vep',$ we have the theorem.
\end{proof}

We now prove Theorem \ref{thm: point to line max limit lower bound}. We need the following estimate from LPP geometry whose proof is deferred to Section \ref{s:lpp}. Let $U_c$ denote the parallelogram whose opposite sides lie on $\cl_{0}$ (resp.\ $\cl_{2N}$) with midpoints $\bo{0}$ (resp.\ $\bo{N}$) and each with length $\lfloor cN^{2/3} \rfloor,$ for some $c>0$. Let us define the following restricted passage time. 
\[
T_N^{**}:=\max_{\gamma: \cl_{0} \rightarrow \bo{N}, \gamma \subset U_c}T(\gamma).
\]
\begin{lemma} 
\label{lemma: restriced passage time large} We have for fixed $\vep, c>0$and for any $t$ sufficiently large (depending on $\vep, c$) and for sufficiently large $N$ (depending on $\vep, t, c$)
\[
\P \left( T_N^{**} \geq 4N + 2^{4/3} t N^{1/3}\right) \geq \frac{1}{2}e^{-\left(\frac 43+\vep \right)t^{3/2}}.
\]

\end{lemma}

\begin{proof}[Proof of Theorem \ref{thm: point to line max limit lower bound}]
    As before, we consider the interval $J^t$ and divide it into interval of length $\lfloor(2N)^{2/3} \rfloor$. Without loss of generality we assume $t$ is an integer and there are $t$ many intervals. For $1 \leq i \leq t,$ let $u_N(s_i)$ be the midpoints of the respective intervals $J_i$'s. We wish to calculate 
    \[
    \P \left( \max_{0 \leq s \leq t}\frac{T_N^*(s)-4N}{2^{4/3}N^{1/3}(\log t)^{2/3}} \leq \left(\frac{3}{4}\right)^{2/3}(1-\vep) \right).
    \]
    Let us consider the following disjoint parallelograms. For $1 \leq i \leq t$
    \[
    U_i:=\Big \{x \in \Z^2: 0 \leq \phi(x) \leq 2N, |\psi(x)-\psi(u_N(s_i))| \leq \lfloor \frac{N^{2/3}}{2} \rfloor \Big \}.
    \]
    For $1 \leq i \leq t,$ let us define
    \[
    T_N^{**}(s_i):=\max_{\gamma:\cl_{0}\rightarrow u_N(s_i),\gamma \subset U_i}T(\gamma).
    \]
   
    \[
    \Big \{\max_{0 \leq s \leq t}\frac{T_N^*(s)-4N}{2^{4/3}N^{1/3}(\log t)^{2/3}} \leq \left(\frac{3}{4}\right)^{2/3}(1-\vep) \Big \} \subset \bigcap_{i=1}^{t}\Big \{\frac{T_N^{**}(s_i)-4N}{2^{4/3}N^{1/3}(\log t)^{2/3}} \leq \left(\frac{3}{4}\right)^{2/3}(1-\vep) \Big \}.
    \]
    
    By Lemma \ref{lemma: restriced passage time large}, taking $c=\frac 12$, for $t$ sufficiently large (depending on $\vep$)
    \[
    \P \left( \frac{T_N^{**}(s_i)-4n}{2^{4/3}N^{1/3}(\log t)^{2/3}} \leq \left(\frac{3}{4}\right)^{2/3}(1-\vep)\right) \leq 1-\frac 12 e^{- (1-\vep^2)^{3/2} \log t}.
    \]
    Hence, by the fact that the distribution of $T_N^*(s)$ does not depend on $s$ and independence we get
    \[
    \P \left( \max_{0 \leq s \leq t}\frac{T_N^*(s)-4N}{2^{4/3}N^{1/3}(\log t)^{2/3}} \leq \left(\frac{3}{4}\right)^{2/3}(1-\vep) \right) \leq \left(1- \frac 12 e^{- (1-\vep^2)^{3/2} \log t}\right)^{t}.
    \]
   
    Now, the right hand side goes to $0$ as $t \rightarrow \infty.$ So, we can chose $t$ large enough so that the right hand side is smaller than $\frac 12$ uniformly in $N.$ This completes the proof. 
\end{proof} 

We now prove Theorem \ref{t:airy1limit}, (ii). Similar to the previous case the theorem will follow from Theorem \ref{thm: point to line min limit upper bound} and Theorem  \ref{thm: point to line min limit lower bound} below.
\begin{theorem}
\label{thm: point to line min limit upper bound}For any $\vep>0$ there exists $\gamma>0$ such that for for sufficiently large $t$ (depending on $\vep$), sufficiently large $N$ (depending on $t, \vep$) 
\[
\P \left( \min_{0 \leq s \leq t}\frac{T_N^*(s)-4n}{2^{4/3}N^{1/3}(\log t)^{1/3}} \leq -\left(6\right)^{1/3}(1+\vep) \right) \leq t^{-\gamma}.
\]
\end{theorem}

\begin{theorem}
\label{thm: point to line min limit lower bound}
    For any $\vep>0,$ for all $t$ sufficiently large (depending only on $\vep$) and for sufficiently large $N$ (depending on $t, \vep$)
\[
\P \left( \min_{0 \leq s \leq t}\frac{T_N^*(s)-4N}{2^{4/3}N^{1/3}(\log t)^{1/3}} \leq -\left(6 \right)^{1/3}(1-\vep) \right) \geq \frac 12.
\]
    
\end{theorem}

\begin{proof}[Proof of Theorem \ref{t:airy1limit}, (ii)]   By ergodicity of $\ca_1$, Theorem \ref{t:lppairy1limit}, Theorem \ref{thm: point to line min limit upper bound}, Theorem \ref{thm: point to line min limit lower bound}, and applying same argument as we did in Theorem \ref{t:airy1limit}, (i)
the theorem follows.  

\end{proof}

For the proof of Theorem \ref{thm: point to line min limit upper bound}, we need the following result which we will prove in Section \ref{s:lpp}.
Let $I^\delta$ denote the interval of length $\lfloor \delta (2N)^{2/3} \rfloor$ on $\cl_{2N}$ with midpoint $\bo{N}$. \begin{lemma}
    \label{lemma: min interval to line}
    For any $\vep>0$, there exists $\delta>0$ such that such that for any $t$ sufficiently large (depending on $\vep$), for $N$ sufficiently large (depending on $t, \vep$) 
    \[
    \P \left( \min_{x \in I^\delta}T_{\cl_{0},x}-4N \leq -t2^{4/3}N^{1/3}\right) \leq e^{-\frac {1}{6}(1-\vep) t^{3}}.
    \]
\end{lemma}

\begin{proof}[Proof of Theorem \ref{thm: point to line min limit upper bound}] The proof is similar to Theorem \ref{thm: point to line max limit upper bound}. We need to apply a union bound and Lemma \ref{lemma: min interval to line}. To avoid repetition we omit the details. 
\end{proof}

\begin{proof}[Proof of Theorem \ref{thm: point to line min limit lower bound}]
    As before, we consider the interval $J^t$ and divide it into interval of length $\lfloor t^{\delta} (2N)^{2/3} \rfloor,$ where we will choose $\delta>0$ small enough depending on $\vep$. Without loss of generality assume $t^{1-\delta}$ is an integer and for $1 \leq i \leq t^{1-\delta},$ let $u_N(s_i)$ be the midpoints of the respective intervals $J_i$'s. We wish to calculate 
    \[
    \P \left( \min_{0 \leq s \leq t}\frac{T_N^*(s)-4N}{2^{4/3}N^{1/3}(\log t)^{2/3}} \leq -\left(6\right)^{1/3}(1-\vep) \right).
    \]
    We define the parallelograms $U_i^t$ as follows.
     \[
    U_i^t:=\Big \{x \in \Z^2: 0 \leq \phi(x) \leq 2N, |\psi(x)-\psi(u_N(s_i))| \leq \frac{t^\delta}{16}  \Big \}.
    \]
    
    Note that the parallelograms are disjoint. For $1 \leq i \leq t^{1-\delta},$ let 
    \[
    T_N^{**}(s_i):=\max_{\gamma:\cl_{0}\rightarrow u_N(s_i),\gamma \subset U_i^t}T(\gamma).
    \]
    Further, we consider the following events.
    \[
    \mathsf{TF}_i:=\Big \{ \Gamma_{\cl_0,u_N(s_i)} \subset U_i^t\Big \},
    \]
    We have 
    \begin{align*}
    &\Big \{\min_{0 \leq s \leq t}\frac{T_N^*(s)-4N}{2^{4/3}N^{1/3}(\log t)^{1/3}} \geq -\left(6\right)^{1/3}(1-\vep) \Big \}\\
    &\subset \left(\bigcap_{i=1}^{t^{1-\delta}}\Big \{\frac{T_N^{**}(s_i)-4N}{2^{4/3}N^{1/3}(\log t)^{1/3}} \geq -\left(6\right)^{1/3}(1-\vep) \Big \}\right) \bigcup \left( \bigcup_{i=1}^{t^{1-\delta}}\mathsf{TF}_i^c\right).
    \end{align*}
    Observe that the events under the intersection are independent. Hence, by Proposition \ref{l:p2llower} and \cite[Theorem 2.3]{BBF22} we have for $t^\delta \ll n^{1/14}$ and $t$ sufficiently large (depending on $\vep$)
    \[
    \P \left( \min_{0 \leq s \leq t}\frac{T_N^{*}(s)-4n}{2^{4/3}N^{1/3}(\log t)^{1/3}} \geq -\left(6\right)^{1/3}(1-\vep)\right) \leq \left(1- \frac 12e^{- (1-\vep^2)^{3} \log t}\right)^{t^{1-\delta}}+t^{1-\delta} \left( e^{ct^{2 \delta}-\frac 16 t^{3 \delta}} \right).
    \]

     Now, we can chose $\delta>0$ small enough (depending on $\vep$) so that the right hand side goes to $0$ as $t \rightarrow \infty.$ So, we can chose $t$ large enough (depending on $\delta$) so that the right hand side is smaller than $\frac 12$ uniformly in $N.$ This completes the proof.
\end{proof}

\section{Extrema of Airy$_2$ process: Proof of Theorem \ref{t:airy2limit}}
\label{s:a2}
In this section, we prove Theorem \ref{t:airy2limit}, (i) and Theorem \ref{t:airy2limit}, (ii). The general structure of the proofs are similar to the ones for the Airy$_1$ process but the specific details of the arguments are somewhat different. Theorem \ref{t:airy2limit}, (i) will follow from the following two theorems. 
\begin{theorem}
\label{t:2maxub}
For any $\varepsilon>0$, there exists $\gamma>0$ such that for sufficiently large $t$ (depending on $\vep$), for all $N$ sufficiently large (depending on $t, \vep$) 
$$\P\left(\max_{0\le s\le t} \frac{T_{N}(s)-4N+s^22^{4/3}N^{1/3}}{2^{4/3}N^{1/3}(\log t)^{2/3}}\ge \left(\frac 34 \right)^{2/3}(1+\varepsilon) \right) \leq t^{-\gamma}.$$
\end{theorem}

\begin{theorem}
    \label{t:2maxlb}
    For any $\varepsilon>0$, there is a constant $c>0$ such that
$$\liminf_{N\to \infty}\P\left(\max_{0\le s\le t} \frac{T_N(s)-4N+s^22^{4/3}N^{1/3}}{2^{4/3}N^{1/3}(\log t)^{2/3}}\ge \left(\frac 34 \right)^{2/3}(1-\varepsilon)\right) \geq c $$
uniformly for all $t$ large (depending on $\vep$).
\end{theorem}

\begin{proof}[Proof of Theorem \ref{t:airy2limit}, (i)]
By ergodicity of $\ca_2(\cdot)$ \cite{PSpng02} we have there exist constants $c_1',c_2'$ (possibly infinite) such that 
\[
\limsup_{t \rightarrow \infty} \frac{\max_{0 \leq s \leq t \ca_2(s)}}{(\log t)^{2/3}}=c_1' \text{ a.s. }, \text{ } \liminf_{t \rightarrow \infty} \frac{\max_{0 \leq s \leq t \ca_2(s)}}{(\log t)^{2/3}}=c_2' \text{ a.s. }.
\]
From Theorem \ref{t:lppairy2limit} and Theorem \ref{t:2maxub} we get $c_1' \leq \left( \frac 34 \right)^{2/3}.$ From Theorem \ref{t:lppairy2limit} and Theorem \ref{t:2maxlb} we get $c_2' \geq \left( \frac 34 \right)^{2/3}.$ This completes the proof.
\end{proof}To prove Theorem \ref{t:2maxub} we need the following lemma which we will prove in Section \ref{s:lpp}. Let $I^{m,\delta}$ denote the interval of length $ \lfloor \delta (2N)^{2/3} \rfloor$ on $\cl_{2N}$ with mid point $\bo{N_m}:=(N-\lfloor m(2N)^{2/3} \rfloor,N+ \lfloor m(2N)^{2/3} \rfloor)$.
\begin{lemma}
    \label{lemma: itp rt ub }
    For any $ \vep>0, m_0>0$ there exists $\delta>0$ (depending on $\vep$) such that for all $m$ with $m \in [0,m_0]$ and and for all $t$ sufficiently large (depending on $\vep$) and for all $N$ sufficiently large (depending on $m_0, \vep, t$)
\[
\P \left( \max_{u_N(x) \in I^{m,\delta}  }\left(T_N(x)-4N+2^{4/3}N^{1/3}x^2\right) \geq 2^{4/3}N^{1/3}t\right) \leq e^{-\frac 43 (1-\vep) t^{3/2}}.
\]

\end{lemma}

\begin{proof}[Proof of Theorem \ref{t:2maxub}]

The proof is similar to the proof of Theorem \ref{thm: point to line min limit upper bound}. Without loss of generality we divide the interval $[0,t]$ into $ \approx \frac t \delta$ many intervals ($\delta$ will be chosen later) each of length $\delta$. Let $I^\delta_i$ denote the interval of length $ \approx \lfloor \delta (2N)^{2/3}\rfloor$ with end points $(N-\lfloor (i-1) \delta (2N)^{2/3} \rfloor,N+\lfloor (i-1) \delta (2N)^{2/3} \rfloor)$ and $(N-\lfloor i \delta (2N)^{2/3} \rfloor,N+\lfloor i \delta (2N)^{2/3} \rfloor).$ Let us define the following events. 
    \[
    A_i:=\Big \{ \max_{u_N(x) \in I_i^{\delta}} \frac{T_N(x)-4N+x^22^{4/3}N^{1/3}}{2^{4/3}N^{1/3}(\log t)^{2/3}} \geq \left( \frac 34 \right)^{2/3}(1+\vep) \Big \}.
    \]
    We have
        \[\Big \{ \max_{0 \leq s \leq t} \frac{T_N(s)-4N+s^22^{4/3}N^{1/3}}{2^{4/3}N^{1/3}(\log t)^{2/3}} \geq \left (\frac 34 \right)^{2/3}(1+\vep) \Big \} \\ \subset \bigcup_{i=1}^{\frac t\delta}A_i.\]
        We look at each event on the right hand side. 
       
Observe that in Lemma \ref{lemma: itp rt ub } if we chose $\delta$ corresponding to $\vep_1$ (where $\vep_1$ is chosen such that $(1+\vep)^{3/2}(1-\vep_1)>1$) then we have for $t$ sufficiently large (depending on $\vep$) and for sufficiently large $N$ (depending on $t, \vep$) 
\[
\P \left( \max_{u_N(x) \in I_i^{\delta}} \frac{T_N(x)-4N+x^2 2^{4/3} N^{1/3}}{2^{4/3}N^{1/3}(\log t)^{2/3}} \geq \left( \frac 34 \right)^{2/3}(1+\vep) \right) \leq Ce^{-(1+\vep)^{3/2}(1-\vep_1)}.
\]
So, taking a union bound we get
\[
\P \left( \max_{0 \leq s \leq t} \frac{T_N(s)-4N+s^22^{4/3}N^{1/3}}{2^{4/3}N^{1/3}(\log t)^{2/3}} \geq \left (\frac 34 \right)^{2/3}(1+\vep) \log t\right) \leq C\frac t\delta e^{-(1+\vep)^{3/2}(1-\vep_1) \log t}.
\]
By choice of $\vep_1,$ the above inequality completes the proof.
\end{proof}

\subsection{Proof of Theorem \ref{t:2maxlb}}
Let $\varepsilon>0$ be fixed. Let us choose $\delta>0$ to be fixed sufficiently small depending on  $\varepsilon$. Let $s_i=it^{\delta}$ and let 
$u_{i}=(N-\lfloor s_i(2N)^{2/3} \rfloor, N+\lfloor s_i(2N)^{2/3} \rfloor)$ for $i=0,1,\ldots t^{1-\delta}$. Clearly it suffices to prove that 

\begin{equation}
    \label{eq:2lbsuff}
    \liminf_{N\to \infty}\P\left(\max_{0\le i\le t^{1-\delta}} \frac{T_N(s_i)-4N+s_i^22^{4/3}N^{1/3}}{2^{4/3}N^{1/3}}\ge \left (\frac 34 \right)^{2/3}(1-\varepsilon)(\log t)^{2/3} \right)> 0
\end{equation}
uniformly for all $t$ large.

Call the event above $A$. We shall construct independent events $B$ and $C$ with $B\cap C \subset A$. 

Without loss of generality let us assume the $\delta N$ is an integer. Let $v_i=(\delta N-\lfloor \delta s_i(2N)^{2/3} \rfloor, \delta N+\lfloor \delta s_i(2N)^{2/3}\rfloor)$. Let us define the event $B$ by 

$$B=\left\{{T_{\mathbf{0},v_i}-4\delta N + \delta^{2/3}s_i^{2}2^{4/3}(\delta N)^{1/3}}\ge -\left (\frac 34 \right)^{2/3}\frac{\varepsilon}{2}2^{4/3}(\log t)^{2/3} N^{1/3}~~\forall i\right\}.$$

We define the event $C$ by 

$$C=\left\{\exists i~~T_{v_i,u_{i}}-4(1-\delta) N +(1-\delta)^{2/3}s_i^22^{4/3}((1-\delta)N)^{1/3} \ge \left (\frac{3}{4}\right)^{2/3}(1-\frac \vep2)2^{4/3}(\log t)^{2/3} 2^{4/3}N^{1/3}\right\}.$$
\begin{figure}[t!]
     \centering
     \includegraphics[width=8 cm]{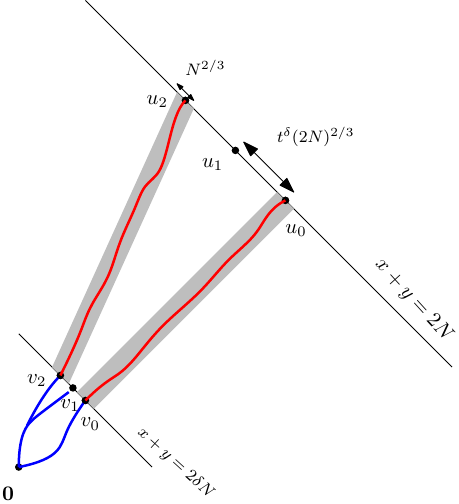}
    \caption{To prove Theorem \ref{t:2maxlb} we consider the points $u_i$ on $x+y=2N$ and $v_i$ on $x+y=2 \delta N$. The passage time from to $\bo{0}$ to some $u_i$ will be large if passage time from $\bo{0}$ to $v_i$ is not too small for all $i$ (this is the event $B$) and passage time from $u_i$ to $v_i$ is large for some $i$ (this is the event $C$). $B$ and $C$ are independent events and we show both of them are positive probability events (Lemma \ref{l:B} and Lemma \ref{l:C}). To find the probability of $C$ we construct independent events $C_i$ where we look at the passage time restricted within disjoint parallelograms and find estimate for the probability of restricted passage time being large (Lemma \ref{l:restricted}). The disjoint parallelograms are shaded in grey in the figure.} 
    \label{fig: A2max} 
     \end{figure}

Since $T_{\mathbf{0},u_i}\ge T_{\bo{0},v_i}+T_{u_i, v_i}$ it follows that for $\varepsilon$ sufficiently small, we have $B\cap C \subset A$ (see Figure \ref{fig: A2max}). 

Since $B$ and $C$ are independent \eqref{eq:2lbsuff} follows from the next two lemmas.  

\begin{lemma}
    \label{l:B}
    We have $\liminf_{N\to \infty}\P(B)>0$
    uniformly in all $t$ large. 
\end{lemma}

\begin{proof}
    For $1 \leq i \leq t^{1-\delta}$, let us define the following events. 
    \begin{itemize}
        \item $B_i:=\big \{ {T_{\mathbf{0},v_i}-4\delta N + \delta^{2/3}s_i^{2}2^{4/3}(\delta N)^{1/3}}\ge -2^{4/3}\frac \varepsilon2(\log t)^{2/3} N^{1/3}\big \}.$
    \end{itemize}
    Then we have 
    \[
    B=\bigcap_{i=1}^{t^{1-\delta}} B_i.
    \]
Since, each $B_i$ is an increasing events, we have 
\[
\P(B) \geq \prod_{i=1}^{t^{1-\delta}}\P(B_i).
\]
We look at $\P(B_i^c).$ By \cite[Theorem 1.4, (ii)]{BBBK24} we have for any $\vep>0$ and for all $(m,n)$ with $\frac mn$ going to $1$ and $m,n$ sufficiently large and $t$ sufficiently large
\begin{equation}
    \label{eq: lower tail upper bound ptp}
    \P\left(T_{\mathbf{0},(m,n)}-(\sqrt{m}+\sqrt{n})^2\le -t(\sqrt{m}+\sqrt{n})^{4/3}(\sqrt{mn})^{-1/3}\right)\le e^{-\frac{1}{12}(1-\varepsilon)t^{3}}.
\end{equation}
Now, if $v_i=(v_{i,1},v_{i,2})$ then a Taylor series expansion shows for $N$ sufficiently large (depending on $t$)
\[
\left( \sqrt{v_{i,1}}+\sqrt{v_{i,2}}\right)^2=4\delta N-s_i^2\delta^{2/3}2^{4/3}(\delta N)^{1/3}+O(\delta N^{-1/3}).
\]
So, we can choose $N$ sufficiently large depending on $\vep, t$ such that 
\[
B_i^c \subset \left \{ {T_{\mathbf{0},v_i}-\left( \sqrt{v_{i,1}}+\sqrt{v_{i,2}}\right)^2}\le -\frac{\varepsilon}{4}2^{4/3}(\log t)^{2/3} N^{1/3}\right \}.
\]
Further, we can choose $N$ sufficiently large (depending on $t$ and $\vep$) such that 
\[
\left( \sqrt{v_{i,1}}+\sqrt{v_{i,2}}\right)^{4/3}\left( \sqrt{v_{i,1}v_{i,2}}\right)^{-1/3} \leq 2^{4/3}(1+\vep)(\delta N)^{1/3}.
\]
Combining all the above, and by \eqref{eq: lower tail upper bound ptp} we get for some constant $c>0,$ 
\[
\P(B_i^c) \leq e^{-c (\log t)^2}.
\]
So, we have 
\[
\P(B) \geq \left (1-e^{-c (\log t)^2} \right)^{t^{1-\delta}}.
\]
Now, as $e^{-c(\log t)^2} \leq \frac{1}{t^{1-\delta}}$
, for $t$ sufficiently large and $N$ sufficiently large (depending on $t$)
\[
\P(B) \geq \left(1-\frac{1}{t^{1-\delta}} \right)^{t^{1-\delta}} \geq c.
\]
This completes the proof.
\end{proof}

\begin{lemma}
    \label{l:C}
    We have $\liminf_{N\to \infty}\P(C)>0$
    uniformly in all $t$ large (possibly depending on $\delta$). 
\end{lemma}

For the proof of Lemma \ref{l:C} we need the following result. For $u=u_N(s)$ as before, let $\widetilde{T}_{\mathbf{0},u}$ denote the weight of the maximum weight path from $\mathbf{0}$ to $u$ contained in the parallelogram $P(u)$ with corners $\pm(N^{2/3},-N^{2/3}), u\pm (N^{2/3},-N^{2/3})$ (i.e., a strip of width $N^{2/3}$ around the line segment joining $\mathbf{0}$ and $u$). We have the following lemma.

\begin{lemma}
    \label{l:restricted}
    Given $\varepsilon>0$, there exists $T_0>0$ such that any $S,T>0$ with $T>T_0$, for all $N>N_0(S,T,\varepsilon)$ and all $s\in [0,S]$ and all $t\in [T_0,T]$ we have 
    $$\P\left(\widetilde{T}_{\mathbf{0},u}\ge 4N-s^22^{4/3}N^{1/3}+t2^{4/3}N^{1/3}\right)\ge e^{-(\frac{4}{3}+\varepsilon)t^{3/2}}.$$
\end{lemma}

Postponing the proof of Lemma \ref{l:restricted} to Section \ref{s:lpp} let us prove Lemma \ref{l:C}.

\begin{proof}[Proof of Lemma \ref{l:C}]
Let $\widetilde{T}_{v_i,u_i}$ denote the maximum weight of a path from $v_i$ to $u_i$ contained in the strip $P_i$ of width $N^{2/3}$ around the straight line joining $v_i$ and $u_i$. Let $C_i$ denote the event that 
    $$\widetilde{T}_{v_i,u_{i}}-4(1-\delta) N +(1-\delta)^{2/3}s_i^22^{4/3}((1-\delta)N)^{1/3} \ge \left(\frac{3}{4}\right)^{2/3}(1-\frac \vep2)2^{4/3}(\log t)^{2/3} 2^{4/3}N^{1/3}.$$
    Clearly, 
    $$C^c\subseteq \cap_{i} C^c_{i}.$$
    Clearly, for $t$ sufficiently large (depending on $\delta$) we know that the strips $P_i$ are disjoint and therefore the events $C_i$ are independent. By Lemma \ref{l:restricted} it follows that for $\delta$ sufficiently small depending on $\varepsilon$, $t$ sufficiently large depending on $\delta$ and $N$ sufficiently large depending on all other parameters we get
    $$\P(C^c_{i})\le 1-t^{-1+2\delta}.$$
    Therefore, using the independence of $C_i$'s we get 
    $\P(C^c)\le(1-t^{-1+2\delta})^{t^{1-\delta}}<1/2$ 
    for $t$ large enough, completing the proof. 
\end{proof}

We now proceed with the proof of Theorem \ref{t:airy2limit}, (ii). As in the previous results this will follow from the following two theorems.

\begin{theorem}
\label{thm:minima of airy_2 upper bound}
For any $\vep>0$ there exists $\gamma>0$ such that for all $t$ sufficiently large (depending on $\vep$) and $N$ sufficiently large (depending on $t, \vep$)
\[
\P \left( \min_{0 \leq s \leq t} \frac{T_N(s)-4N+s^22^{4/3}N^{1/3}}{2^{4/3}N^{1/3}(\log t)^{1/3}} \leq -(12)^{1/3}(1+\vep)\right) \leq t^{-\gamma}.
\]
\end{theorem}

\begin{theorem}
    \label{t:2minlb}
    For any $\varepsilon>0$,
$$\liminf_{N\to \infty}\P\left(\min_{0\le s\le t} \frac{T_N(s)-4N+s^22^{4/3}N^{1/3}}{2^{4/3}N^{1/3}(\log t)^{1/3}}\le -(12)^{1/3}(1-\vep) \right)> 0 $$
uniformly for all large $t$ (depending on $\vep$).
\end{theorem}

\begin{proof}[Proof of Theorem \ref{t:airy2limit}, (ii)]
 By ergodicity of $\ca_2(\cdot)$, Theorem \ref{t:lppairy2limit}, Theorem \ref{thm:minima of airy_2 upper bound}, and Theorem \ref{t:2minlb} the proof follows. 
\end{proof}

To prove Theorem \ref{thm:minima of airy_2 upper bound} we need the following lemma which will also be proved in Section \ref{s:lpp}. Recall the setup of Lemma \ref{lemma: itp rt ub }. 

\begin{lemma} 
\label{lemma: itp lt ub}For any $ \vep>0, m_0>0$ there exists $\delta>0$ (depending on $\vep$),such that for $m \in [0,m_0]$ and for all $t$ sufficiently large (depending on $\vep$), $N$ sufficiently large (depending on $m_0,\vep, t$) 
\[
\P \left( \min_{u_N(x) \in I^{m,\delta}  }\left(T_N(x)-4N+2^{4/3}N^{1/3}x^2\right) \leq -2^{4/3}N^{1/3}t\right) \leq e^{-\frac{1}{12}(1-\vep) t^3}.
\]
\end{lemma}
\begin{proof}[Proof of Theorem \ref{thm:minima of airy_2 upper bound}] By Lemma \ref{lemma: itp lt ub}, applying a union bound and arguing similarly as in the proof of Theorem \ref{t:2maxub}, the result follows.
\end{proof}

\subsection{Proof of Theorem \ref{t:2minlb}}
\begin{figure}[t!]
     \centering
     \includegraphics[width=10 cm]{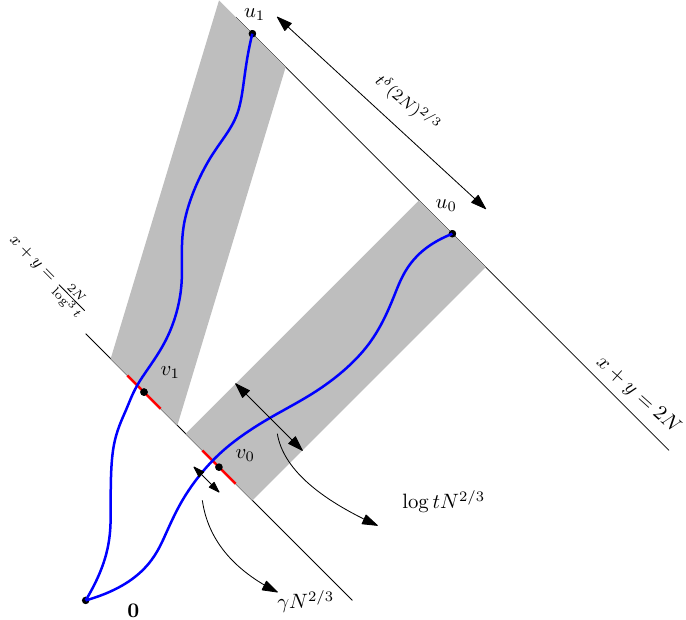}
    \caption{To prove Theorem \ref{t:2minlb} we consider the above figure. We consider the lines $x+y=2 \kappa N$ and $x+y=2N$ and the points $u_i.$ We consider the following events. For all $i, \Gamma_{\bo{0},u_i}$ intersect $x+y=2 \kappa N$ inside $\wt{I}_i$ (marked in red in the figure, these are the events $C_i$). For all $i, \Gamma_{\bo{0},u_i}$ stays within the parallelogram $P_i$ (shaded grey in the figure, these are the events $D_i$). For all $i,$ the passage time from $\bo{0}$ to $\wt{I}_i$ is not too large (these are the events $A_i$). There exists an $i$ such that restricted passage time within $P_i$ from $\wt{I}_i$ to $u_i$ is small (these are the events $B_i$). Observe that on intersection of all the four events above the passage time from $\bo{0}$ to $u_i$ is small for some $i$. We show that the intersection of these four events is a positive probability event. The parallelogram $P_i$'s are constructed in such a way that they are disjoint. We use this to estimate $\P(B).$} 
    \label{fig: A2min} 
     \end{figure}
The general structure of this proof is similar to the proof of Theorem \ref{t:2maxlb} but we need to be a bit more careful about the parameter choices. Recall the points $u_i$ from the proof of Theorem \ref{t:2maxlb}; we shall use the same notations here. But our definition of the points $v_i$ will be different. Let $\kappa=(\frac{1}{\log t})^{3}$ and without loss of generality (to avoid notational complications) we shall also assume that $\kappa N$ is an integer. Let $v_i=(\kappa N-\lfloor \kappa s_i(2N)^{2/3} \rfloor, \kappa N+\lfloor \kappa s_i(2N)^{2/3}\rfloor)$. That is, the points $v_i$ are (essentially) the points where the straight line joining $\mathbf{0}$ and $u_i$ intersects the line $x+y=2\kappa N$. Let $I_{i}$ denote the interval on the line $x+y=2\kappa N$ with length $2\log t N^{2/3}$ and midpoint $v_i$, and let $\widetilde{I}_i$ denote the interval with the same midpoint but with length $\gamma N^{2/3}$ (for some $\gamma$ small enough to be chosen later). Since $s_i=it^{\delta}$ and $\kappa t^{\delta}\gg \log t$ for $t$ large it follows that for $t$ sufficiently large the intervals $I_{i}$ are disjoint. Let $P_i$ denote the parallelogram whose shorter pair of opposite sides have midpoints $v_i$ and $u_i$ respectively, and the side containing $v_i$ coincides with $I_{i}$, as before the parallelograms $P_i$ are also disjoint (see Figure \ref{fig: A2min}). For $v\in \widetilde{I}_{i}$, let $\widetilde{T}_{v,u_i}$ denote the maximum weight of a path from $v$ to $u_i$ contained in $P_{i}$. We now define a series of events.

For $i=0,1,2,\ldots, t^{1-\delta}$, let $A_{i}$ denote the event that 
$$A_{i}=\left\{\max_{v\in \widetilde{I}_i}\left \{ T_{\mathbf{0},v}-\E T_{\mathbf{0},v}\right \}\le \frac{\varepsilon}{100}(\log t)^{1/3}N^{1/3}\right\}.$$
Let us set $A=\cap A_{i}$. Next, let $B_{i}$ denote the event that 
$$B_{i}=\left\{\max_{v\in \widetilde{I}_i} \left \{\wt{T}_{v,u_i}-\E T_{v,u_i}\right\}\le -(12)^{1/3}(1-\varepsilon/10)(\log t)^{1/3}2^{4/3}N^{1/3}\right\}$$
and set $B=\cup B_i$. Let $C_{i}$ denote the event that the intersection of $\Gamma_{\mathbf{0},u_i}$ with $x+y=2\kappa N$ is contained in $\widetilde{I}_i$, and set $C=\cap_i C_{i}$. Finally let $D_{i}$ denote the event that the part of $\Gamma_{\bo{0},u_i}$ above the line $x+y=2\kappa N$ is contained in $P_{i}$, and let us set $D=\cap_{i}D_{i}$.

We have the following lemma. 

\begin{lemma}
    \label{l:events}
    For $t$ sufficiently large, and $N$ sufficiently large depending on $t$, we have, on $A\cap B\cap C\cap D$, there exists $i$ such that 
    $$\frac{T_N(s_i)-4N+s_i^22^{4/3}N^{1/3}}{2^{4/3}N^{1/3}}\le -(12)^{1/3}(1-\vep)(\log t)^{1/3}.$$
\end{lemma}

\begin{proof}
    On the event $A \cap B \cap C \cap D$ we have there exists an $i$ such that 
    \[
    \max_{v\in \widetilde{I}_i} \left \{ \wt{T}_{v,u_i}-\E T_{v,u_i} \right \}\le -(12)^{1/3}(1-\varepsilon/10)(\log t)^{1/3}2^{4/3}N^{1/3}.
    \]
Further, on the event $A \cap B \cap C \cap D$ there exists $v \in \wt{I_i}$ such that 
\[
T_{\bo{0},u_i}=T_{\bo{0},v}+\wt{T}_{v,u_i}.
\]
Hence, $A \cap B \cap C \cap D$ we have for $N$ sufficiently large depending on $t$ and $\vep$ and using the fact that $\E T_{\bo{0},u_i} \geq \E T_{\bo{0},v}+\E T_{v,u_i}$ 
\[
T_{\bo{0},u_i} \leq \E T_{\bo{0},u_i} -(12)^{1/3}(1-\frac \vep2)(\log t)^{1/3}2^{4/3}N^{1/3} \leq 4N-s_i^22^{4/3}N^{1/3}-(12)^{1/3}(1-\vep)(\log t)^{1/3}2^{4/3}N^{1/3},
\]
where the last inequality follows from the fact that (see \cite[Theorem 2]{LR10})  for any $\gamma>0$ with $\gamma < \frac mn < \gamma^{-1},$ there exists constant $C>0$ (depending only on $\gamma$) such that for all $n \geq 1$
\begin{equation}
\label{eq: expectation estimate}
|\E(T_{(0,0),(m,n)})-( \sqrt{m}+\sqrt{n})^2| \leq Cn^{1/3}.
\end{equation}

This completes the proof. 
\end{proof}

Therefore it only remains to show that 
$$\liminf_{t\to \infty}\liminf_{N\to \infty} \P(A\cap B\cap C\cap D)>0.$$
For this we shall need the following two lemmas.

\begin{lemma}
    \label{l:Abound}
    For $t$ sufficiently large and $N$ sufficiently large depending on $t$, we have 
    $\P(A)\ge \frac{1}{4}$. 
\end{lemma}

\begin{proof}
    Since, $A_i$'s are decreasing events by FKG inequality we have 
    \[
    \P(A) \geq \prod_{i=0}^{t^{1-\delta}}\P(A_i).
    \]
    Next we have, if we break $\wt{I}_i$ into $\gamma (\log t)^2$ many intervals each having length $\left( \frac{N}{\log^3 t}\right)^{2/3},$ then by applying \cite[Theorem 4.2, (ii)]{BGZ21} to each of these intervals and by a union bound we have there exists $c>0$ such that for sufficiently large $N$ and $t,$
    \[
    \P(A_i^c) \leq (\log t)^2e^{-c ({\log t})^2}.
    \]
    Hence, there exists $c'>0$ such that for sufficiently large $t$
    \[
    \P(A) \geq \left(1- e^{-c'(\log t)^2}\right)^{t^{1-\delta}} \geq \left (1-\frac{1}{t^{1-\delta}} \right)^{t^{1-\delta}} \geq \frac 14.
    \]
    This completes the proof.
\end{proof}

We also need the next lemma whose proof is postponed to Section \ref{s:lpp}. Recall the setup of Lemma \ref{lemma: itp rt ub }.

\begin{lemma} 
\label{lemma: itp lt lb}For any $ \vep,m_0>0$ there exists $\delta>0$ (depending on $\vep$),and for all $m$ with $m \in [0,m_0]$, for all $t$ sufficiently large (depending on $\vep$), for all $N$ sufficiently large (depending on $m_0, t, \vep$)
\[
\P \left( \max_{u_N(x) \in I^{m,\delta}  }\left \{T_N(x)-\E \left(T_N(x)\right)\right \}\leq -2^{4/3}N^{1/3}t\right) \geq e^{-\frac{1}{12}(1+\vep) t^3}.
\]
    
\end{lemma}

We use the ext lemma to control the probability of $B$. 

\begin{lemma}
\label{l:bbound}
   If $\vep'$ is sufficiently small compared to $\varepsilon$ we have, for each $i$, $\P(B_i)\ge \frac{1}{t^{1-\vep'}}$. 
\end{lemma}
\begin{proof}
    Let us consider the following events. 
    \[
    B_i':=\left\{\max_{v\in \widetilde{I}_i} T_{v,u_i}-\E T_{v,u_i}\le -(12)^{1/3}(1-\varepsilon/10)(\log t)^{1/3}2^{4/3}N^{1/3}\right\}.
    \]
    \[
    \mathsf{LTF}_i:=\{\exists v \in \wt{I}_i \text{ such that } \Gamma_{v,u_i} \text{ goes out of } P_i\}.
    \]
    We have 
    \[
    B_i' \subset B_i \cup \mathsf{LTF}_i.
    \]
    Hence, 
    \[
    \P(B_i) \geq \P(B_i')-\P(\mathsf{LTF}_i).
    \]
    Now, by Lemma \ref{lemma: itp lt lb} we have there exist $\gamma, \vep'>0$ (depending on $\vep$) such that  
    \[
    \P(B_i') \geq \frac{1}{t^{1-\vep'}}.
    \]
    Further, by \cite[Proposition C.9]{BGZ21} we have 
    \[
    \P(\mathsf{LTF}_i) \leq e^{-c (\log t)^3}.
    \]
    Hence, choosing $t$ large enough depending on $\vep'$ we get the lemma. 
\end{proof}

We can now prove Theorem \ref{t:2minlb}. 

\begin{proof}[Proof of Theorem \ref{t:2minlb}]
    Observe that since the events $B_i$ are independent it follows from Lemma \ref{l:bbound} that if $\delta$ is sufficiently small compared to $\varepsilon$, we have, by Lemma \ref{l:bbound}, for large $t$  
    $$\P(B)\ge 1-\prod_{i=1}^{t^{1-\delta}} \P(B_i^c) \ge \frac{1}{2}.$$
    Since $A$ and $B$ are independent it follows from Lemma \ref{l:Abound} and Lemma \ref{l:bbound} that $\P(A\cap B)\ge \frac{1}{8}$. From the choice of $\kappa$ we have on the event $C_i^c$, on $x+y=2 \kappa n,\Gamma_{\bo{0},u_i}$ will have transversal fluctuation more than $\gamma (\log t)^2$. By \cite[Proposition 2.1]{BBB23} it follows that $\P(C_i)\ge 1-e^{-c{(\gamma)^3 (\log t)^6}}$. Hence $\P(C)\ge 99/100$ for $t$ sufficiently large (depending on $\gamma$). By \cite[Proposition C.9]{BGZ21} it also follows that $\P(D_i)\ge 1-e^{-c (\log t)^3} \geq 1-\frac{1}{t^2}$ for $t$ sufficiently large and hence $\P(D)\ge 99/100$. Combining these by a union bound we get $\P(A\cap B\cap C\cap D)\ge 1/10$ and this completes the proof of the theorem by Lemma \ref{l:events}. 
\end{proof}

\section{LPP Estimates}
\label{s:lpp}
In this section we prove the LPP results that we have used throughout the previous sections. We start with the proof of Lemma \ref{lemma: restriced passage time large}. 
\begin{proof}[Proof of Lemma \ref{lemma: restriced passage time large}]Using the fact that point-to-line passage time is larger than the point-to-point passage time, the lemma is a special case of Lemma \ref{l:restricted} when $s=0.$ Therefore, we refer to the proof of Lemma \ref{l:restricted}.
\end{proof}

Next we prove Lemma \ref{lemma: min interval to line}. Although, the upper bound for the right tail of maximum point-to-line passage time over a small interval was known (Lemma \ref{lemma: itl rt ub}), the minimum case considered here is new. We will prove the point-to-point variant of the lemma as well (see Lemma \ref{lemma: itp lt ub}).

\begin{figure}[t!]
    \includegraphics{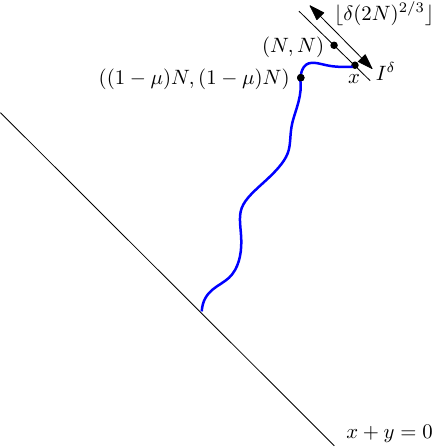}
    \caption{To prove Lemma \ref{lemma: min interval to line} we consider the point $(1-\mu) \bo{N}$ and $\delta=(\mu)^{2/3}.$ Now, if there is $x \in I^\delta$ such that $T_{\cl_0,x}$ is small then by super-additivity, either of the following event happens. Either passage time between $\cl_0$ and $(1-\mu) \bo{N}$ is small (this is the event $\ca$) or passage time between $(1-\mu)\bo{N}$ and $x$ is small (this is the event $\cb$). In the figure two blue paths show the geodesics between $\cl_0$ and $(1-\mu)\bo{N}$ and between $(1-\mu)\bo{N}$ and $x$. By applying Proposition \ref{l:p2llower} to the event $\ca$ we get the optimal exponent in the upper bound. Finally, choosing $\mu$ small enough depending on $\vep$ and applying \cite[Proposition 4.2, (i)]{BGZ21} to $\cb$ we can show that $\P(\cb)$ can be made much smaller than the desired upper bound.  }
    \label{fig: min point to line}
\end{figure}
\begin{proof}[Proof of Lemma \ref{lemma: min interval to line}]
    We consider $\mu$ small enough (to be chosen later) and consider the point $(1-\mu) \bo{N}$ (without loss of generality assume that $(1-\mu)N$ is an integer). Let $\delta=(\mu)^{2/3}.$ We define the following events. Let $\vep':=1-\left(\frac{1-\vep}{1-\frac \vep2}\right)^{1/3}.$\\
$\ca:=\{T_{\cl_{0},(1-\mu)\bo{N}}-4(1-\mu)N \leq -(1-\vep')t2^{4/3}N^{1/3}\}.$\\
$\cb:=\{\min_{x \in I^\delta}\left \{T_{(1-\mu)\bo{N},x}-4 \mu N \right \} \leq -\vep't 2^{4/3}N^{1/3}\}.$
Then (see Figure \ref{fig: min point to line})
\[
\Big \{ \min_{x \in I^\delta}\left \{T_{\cl_{0},x}-4N \right \} \leq -t2^{4/3}N^{1/3} \Big \} \subset \ca \cup \cb.
\]
By Proposition \ref{l:p2llower} we have for $t$ sufficiently large (depending on $\vep'$) and sufficiently large $N$ (depending on $t, \vep'$) 
\[
 \P( \ca) \leq e^{- \frac {1}{6}\frac{(1-\vep')^3(1-\frac \vep2)}{1-\mu}t^3} \leq e^{-\frac{1}{6}(1-\vep) t^3},
\]
where the last inequality comes from choice of $\vep'.$ For the event $\cb,$ we have by \cite[Theorem 4.2(i)]{BGZ21}, there exists $C,c>0$ such that 
\[
\P(\cb) \leq Ce^{-\frac{\vep'^3t^3}{\mu}}.
\]
Now, we can choose $\delta$ (and hence $\mu$) small enough (depending on $\vep', \vep$) so that 
\[
 Ce^{-\frac{\vep'^3t^3}{\mu}} \leq e^{- \frac{1}{6}(1-\vep) t^3}.
\]
This completes the proof. 
\end{proof}

Next we prove Lemma \ref{lemma: itp rt ub }. A version of the lemma without the optimal constant was obtained in \cite[Theorem 4.2, (ii)]{BGZ21}. For the point-to-line case an analogous lemma is Lemma \ref{lemma: itl rt ub}.
\begin{proof}[Proof of Lemma \ref{lemma: itp rt ub }]

    \begin{figure}[t!]
        \includegraphics{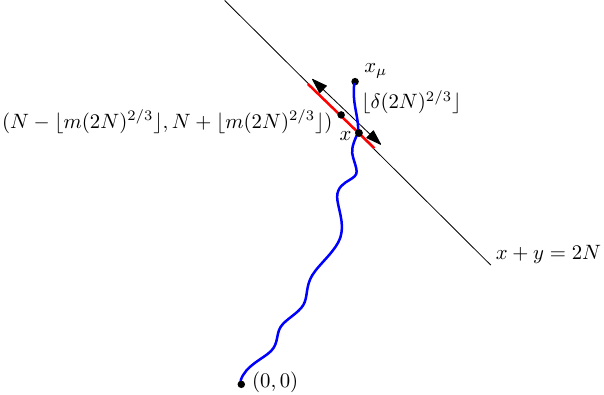}
        \caption{To prove Lemma \ref{lemma: itp rt ub }, we consider the point $x_{\mu}$ which is (essentially) the intersection point of the straight line joining $\bo{0}$ and $\bo{N}_m$ with $\cl_{2(1+\mu)N}$. We consider the interval $I^{m,\delta}$ (marked red in the figure) and choose $\delta=(\mu)^{2/3}$. If there is $x \in I^{m,\delta}$ such that $T_{\bo{0},x}$ is large, then either of the following event happens. Either, $T_{\bo{0}, x_\mu}$ is large (this is the event $\cc$) or $T_{x, x_\mu}$ is small (this is the event $\cd$). In the figure two blue paths are shown which denote the geodesics between $\bo{0}$ and $x$ and between $x$ and $x_\mu$. We apply Proposition \ref{l:p2pupper} to $\cc$ and choose $\mu$ small enough to get the optimal exponent in the upper bound. To show that $\P(\cd)$ can be made much smaller than the desired upper bound, we apply \cite[Proposition 4.2, (i)]{BGZ21} and choose $\mu$ small enough.}
        \label{fig:max interval to point}
    \end{figure}
 We consider $\mu>0$ small enough (to be fixed later). Let us consider the point $x_\mu:=\left ( (1+\mu) N-(1+\mu)m(2N)^{2/3},(1+\mu)N+(1+\mu)m(2N)^{2/3}\right)$ (without loss of generality assume that this belongs to $\Z^2$) (see Figure \ref{fig:max interval to point}). Let $\delta=(\mu)^{2/3}$. We consider the following two events for some $\vep'>0$ to be chosen later based on $\vep$ and $\mu$ to be chosen later depending on both $\vep$ and $\vep'$. 
    \begin{itemize}
        \item $\cc:=\big \{ T_{\bo{0},x_\mu} \geq \E \left(T_{\bo{0},x_\mu} \right)+(1-2\vep')2^{4/3} N^{1/3} t \big \}.$
        \item $\cd:=\big \{ \min_{u_N(x) \in I^{m, \delta}}\left(T_{u_N(x),x_{\mu}}-\E \left(T_{u_N(x),x_\mu} \right) \right)\leq-2^{4/3}\vep'N^{1/3} t\big \}.$
    \end{itemize}
   
As the slope of the line joining $\bo{0}$ to $\bo{N}_m$ is uniformly bounded away from $0$ and $\infty,$ from \eqref{eq: expectation estimate} and by a Taylor expansion we have that there exists a constant $c$ (depending on $\delta$) such that for all $x \in I^{m,\delta}$. 
\[
\E \left(T_N(x)\right)+\E \left(T_{u_N(x),x_\mu} \right) \geq \E \left( T_{\bo{0},x_\mu}\right)-cN^{1/3}.
\]
Hence, we can choose $t$ sufficiently large depending on $\delta$ such that 
    We have 
    \[
     \left \{ \max_{u_N(x) \in I^{m,\delta}  }\left(T_N(x)-\E \left(T_N(x) \right) \right)\geq 2^{4/3}N^{1/3}t \right \} \subset \cc \cup \cd.
    \]
    First we look at the event $\cc.$ We observe that if $x_\mu=(x_{\mu,1},x_{\mu,2}),$ 
    
we have by \eqref{eq: expectation estimate} for sufficiently large $t$ (depending on $\vep'$), and for sufficiently large $N$ (depending on $m_0$)
    \[
    \P(\cc) \leq \P \left(T_{\bo{0},x_\mu}- \left(\sqrt{x_{\mu,1}}+\sqrt{x_{\mu,2}} \right)^2 \geq \left (1-3\vep' \right )2^{4/3}N^{1/3}t\right).
    \]
    Further, if 
    \[
    h_\mu= \left( \sqrt{x_{\mu,1}}+\sqrt{x_{\mu,2}}\right)^{4/3} (\sqrt{x_{\mu,1}x_{\mu,2}})^{-1/3},
    \]
    Then we can choose $N$ sufficiently large (depending on $m_0, \vep'$) such that
    \[
    h_\mu \left (1-3\vep' \right) \leq 2^{4/3}(1+\mu)^{1/3}N^{1/3}.
    \]
    Combining all the above we have 
    \[
    \P(\cc) \leq \P \left(T_{\bo{0},x_\mu}- \left(\sqrt{x_{\mu,1}}+\sqrt{x_{\mu,2}} \right)^2 \geq \left (1-3\vep' \right )^2\frac{h_\mu}{(1+\mu)^{1/3}} t\right).
    \]
    So, by \cite[Theorem 1.1, (ii)]{BBBK24} we get for  sufficiently large $t$ (depending on $\vep'$) and sufficiently large $N$ (depending on $t,\vep'$)
    \[
    \P(\cc) \leq e^{-\frac 43 \left((1-3\vep')^4\frac{t^{3/2}}{\sqrt{1+\mu}} \right)}.
    \]
    Now, we chose $\vep'$ small enough (depending on $\vep$) and $\mu$ small enough (depending on $\vep$) so that the following holds. 
    \[
    \P(\cc) \leq e^{-\frac 43 (1-\vep)t^{3/2}}.
    \]
     Now, we look at the event $\cd.$ 
     By \cite[Theorem 4.2, (i)]{BGZ21} we have for sufficiently large $N$
    \[
    \P(\cd) \leq e^{-c\frac{\vep'^3 t^3}{2\mu}}. 
    \]
    Finally we can choose $\mu$ small enough (depending on $\vep, \vep'$) so that 
    \[
    \P(\cd) \leq e^{-\frac 43 (1-\vep)t^{3/2}}.
    \]
    So, we get for $t$ sufficiently large (depending on $\vep$) and $N$ sufficiently large (depending on $m_0, t, \vep$)
    \[
    \P \left( \max_{u_N(x) \in I^{m,\delta}  }\left(T_N(x)-\E \left(T_N(x) \right) \right)\geq 2^{4/3}N^{1/3}t\right) \leq 2e^{-\frac 43(1-\vep)t^{3/2}}.
    \]
    Finally, for all $\gamma>0$ we can choose $N$ sufficiently large depending on $m_0, t$ and $\gamma$ such that using \eqref{eq: expectation estimate} we have 
    \[
    \P \left(\max_{u_N(x) \in I^{m,\delta}} \left(T_N(x)-4N+2^{4/3} x^2 N^{1/3}\right) \geq (1+\gamma)2^{4/3}N^{1/3}t\right) \leq 2e^{-\frac 43 (1-\vep)t^{3/2}}.
    \]
    This completes the proof.
\end{proof}
Now we provide the proof of Lemma \ref{l:restricted}. Similar estimate as in Lemma \ref{l:restricted} was proved for restricted passage time between $\bo{0}$ and $\bo{N}$ without the optimal constant in \cite[Lemma 4.9]{BGZ21}. For the proof we use the following result about tails of passage times \cite[Theorem 1.3, (ii)]{BBBK24}: given $\varepsilon>0$, there exists $t_0>0$ and $\kappa>0$ such that for all $m,n$ large with $m/n$ bounded away from $0$ and $\infty$ and all $t\in (t_0,n^{\kappa})$ we have 

\begin{equation}
    \label{eq:tail1}
    \P\left(T_{\mathbf{0},(m,n)}-(\sqrt{m}+\sqrt{n})^2\ge t(\sqrt{m}+\sqrt{n})^{4/3}(\sqrt{mn})^{-1/3}\right)\ge e^{-(4/3+\varepsilon)t^{3/2}}.
\end{equation}

\begin{proof}[Proof of Lemma \ref{l:restricted}]
\begin{figure}[t!]
    \includegraphics{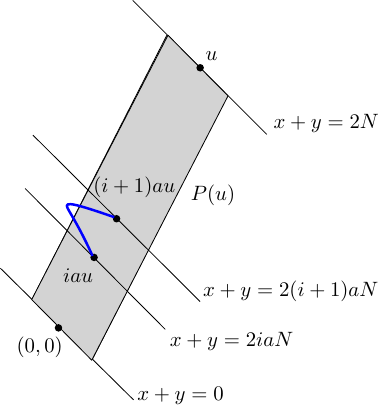}
    \caption{To prove Lemma \ref{l:restricted} we consider the lines $\cl_{2iaN}$. Now, by super-additivity we have $\wt{T}_{\bo{0},u} \geq \sum\wt{T}_{iau, (i+1)au}$ and the passage times $\wt{T}_{iau, (i+1)au}$ are independent. So, if all the passage times $\wt{T}_{iau, (i+1)au}$ are large then $\wt{T}_{\bo{0},u}$ is also large. Now, we can estimate the probability that the restricted passage times between $iau$ and $(i+1)au$ are large by the probability that $T_{iau, (i+1)au}$ is large. This can be done using the fact that, the probability that the geodesic between $iau, (i+1)au$ goes out of $P(u)$(shaded in grey in the figure) is much smaller than the desired probability due to large transversal fluctuation.}
    \label{fig:Restricted passage time large}
\end{figure}
Let us fix $s$ and $t$ as in the statement of the lemma. Let $a>0$ be a small constant to be fixed later (depending on $S$ and $T$). Without loss of generality let us also assume that $aN$ and $a^{-1}$ are both positive integers. Let us also assume without loss of generality that $au\in \Z^2$ (recall that $u=(N-\lfloor s(2N)^{2/3} \rfloor, N+\lfloor(2N)^{2/3} \rfloor$).
The reader can easily check that the argument below can be made to work without these assumptions under minor modifications. Locally, let us denote by $\widetilde{T}$ the constrained passage times within $P(u)$. Clearly, if 
$$\widetilde{T}_{iau, (i+1)au}\ge 4aN-as^22^{4/3}N^{1/3} +at(2^{4/3}N^{1/3})$$
for all $i$, then the event in the statement of the lemma holds (see Figure \ref{fig:Restricted passage time large}). Call the above event $D_{i}$, clearly $D_i$'s are independent. Observe that, by the transversal fluctuation estimates \cite[Proposition C.9]{BGZ21} we know that 
$$\P(D_i)\ge \P(D'_{i})-e^{-ca^{-2}}$$
for some $c>0$ where $D'_{i}$ is the event $D_{i}$ with $\widetilde{T}$ replaced by $T$. Rewriting the event $D_{i}$ as 
$$\widetilde{{T}}_{iau, (i+1)au}\ge 4aN-as^22^{4/3}N^{1/3} +a^{2/3}t(2^{4/3}(aN)^{1/3});$$
it follows from \eqref{eq:tail1} that for sufficiently large $N$ (depending on $S, T$)
$$\P(D_i)\ge e^{-(4/3+\varepsilon)at^{3/2}}-e^{-ca^{-2}}$$
as long as $at^{3/2}\gg 1$. Now, choose $a$ such that $a^{-2}\gg at^{3/2}\gg 1$ such that 
$$\P(D_i)\ge  e^{-(4/3+\varepsilon)at^{3/2}}(1-a),$$ and using the fact that $D_i$s are independent we get 
$$\P(\cap D_i)\ge  \left (e^{-(4/3+\varepsilon)at^{3/2}}(1-a)\right)^{1/a},$$
now by chosing $a$ sufficiently small (depending on $T$) we can see that $(1-a)^{1/a}$ is uniformly lower bounded by a positive constant. 
This completes the proof. 
\end{proof}
\begin{proof}[Proof of Lemma \ref{lemma: itp lt ub}]
The proof is similar to Lemma \ref{lemma: min interval to line}. We choose $\vep'$ small enough (to be chosen later depending on $\vep$) and $\mu$ small enough (to be chosen later depending on both $\vep$ and $\vep'$) and consider the line $\cl_{2 (1-\mu) N}$ (without loss of generality assume $(1-\mu)N$ is an integer). We set $\delta:=(\mu)^{2/3}$. We consider the point \\$x_\mu:=\left ((1-\mu)N-(1-\mu)m(2N)^{2/3},(1-\mu)N+(1-\mu)m(2N)^{2/3} \right). $ We define the following events.
    \begin{itemize}
        \item $\cc':=\big \{ T_{\bo{0},x_\mu} \leq \E \left(T_{\bo{0},x_\mu}\right)-(1-2\vep')2^{4/3} N^{1/3} t \big \}.$
        \item $\cd':=\big \{ \min_{u_N(x) \in I^{m, \delta}}T_{x_{\mu},u_N(x)x} \leq \E \left(T_{x_{\mu},u_N(x)} \right)-2^{4/3}\vep'N^{1/3} t\big \}.$
    \end{itemize}
    As we have argued before in Lemma \ref{lemma: itp rt ub } we get for large $t$ (depending on $\delta$ and $\vep'$) 
    \[
    \left \{ \min_{u_N(x )\in I^{m,\delta}  } \left \{T_N(x)-\E \left(T_N(x) \right) \right \} \leq -2^{4/3}N^{1/3}t\right\} \subset \cc' \cup \cd'.
    \]
    We find upper bounds for $\P(\cc')$ and $\P(\cd')$ exactly in the same way as we did in proof of Lemma \ref{lemma: itp rt ub }. To avoid repetition we omit the details.
\end{proof}

\begin{proof}[Proof of Lemma \ref{lemma: itp lt lb}]

This lemma was already proved in case of (1,1) direction in \cite[Lemma 7.2]{BBBK24}. The proof in the general direction is essentially the same, we provide the details below for completeness. Let us define the following events. $\vep'$ will be chosen later depending on $\vep$ and $\mu$ is small enough (to be chosen later depending on both $\vep$ and $\vep'$) and $\delta=(\mu)^{2/3}.$ Consider the point (without loss of generality assume $(1+\mu)N$ is an integer). $$x_\mu:=\left((1+\mu)N-(1+\mu)m(2N)^{2/3}, (1+\mu)N+(1+\mu) m (2N)^{2/3}\right).$$ 

We now define the following events.
\begin{itemize}
    \item $\cc'':=\left \{ T_{\bo{0},x_\mu}-\E \left(T_{\bo{0},x_\mu} \right)\leq -2^{4/3}N^{1/3}(1+\vep')t\right \},$
    \item $\cd'':=\left \{\min_{u_N(x) \in I^{m,\delta}} \left \{T_{u_N(x),x_\mu}-\E \left( T_{u_N(x),x_\mu}\right)\right \} \geq -2^{4/3}N^{1/3}\frac{\vep'}{2}t \right \}.$
\end{itemize}
As we have argued before in the proof of Lemma \ref{lemma: itp rt ub },
we can choose $t$ large enough depending on $\delta$ such that
\[
\cc'' \cap \cd'' \subset \left \{\max_{u_N(x) \in I^{m,\delta}} \left(T_N(x)-\E \left(T_{N}(x)\right)\right) \leq -2^{4/3}N^{1/3}t\right \}.
\]

By the calculation we did earlier, for $t$ sufficiently large (depending on $\vep'$), for $N$ sufficiently large depending on $t, m_0$ and $\vep'$ we have by \cite[Theorem 1.5, (ii)]{BBBK24} 
\[
\P(\cc'') \geq e^{-\frac{1}{12}\frac{(1+\vep')^4}{1+\mu}t^3}.
\]
We chose $\vep'$ and $\mu$ small enough (depending on $\vep$) so that
\[
\P(\cc'') \geq e^{-\frac{1}{12}(1+\vep)t^3}.
\]
For the event $\cd'',$ by \cite[Theorem 4.2, (i)]{BGZ21} we have 
\[
\P((\cd'')^c) \leq e^{-c\frac{\vep'}{\mu} t^3}.
\]
Now, we have 
\[
\P(\cc'' \cap \cd'') \geq \P(\cc'')-\P((\cd'')^c).
\]
Finally choosing $\mu$ small enough depending on $\vep'$ and on $\vep$ 
proves the lemma. 
\end{proof}
\section{Weak Convergence to $\text{Airy}_1$ Process}
\label{s:weak}
\begin{proof}[Proof of Theorem \ref{t:lppairy1limit}]
    We already have the convergence in sense of finite dimensional distributions \cite{BFPS07,S05,CPF12}. So, by \cite[Theorem 7.1, Theorem 7.3]{B99} it is enough to prove the following the following tightness result: For any $M, \gamma, \lambda>0$ there exists $\delta>0$ such that for $n$ sufficiently large we have 
    \[
    \P \left( \sup_{x,y: |t_1|,|t_2| \leq M, |t_1-t_2|< \delta }|T_n^*(t_1)-T_n^*(t_2)| \geq \lambda 2^{4/3}n^{1/3}\right) < \gamma.
    \]
The proof follows similar idea as \cite[Theorem 3.8]{BGZ21}. We fix $\delta>0$ small enough (to be chosen later). Without loss of generality we divide the interval $[-M,M]$ into $\approx \frac M\delta$ many sub-intervals each having equal length $2 \delta$ and denote these intervals by $I_i^\delta.$ We apply Lemma \ref{lemma: local change of ptl passage time} below to each of these intervals. We take $\vep=2/3, \ell=\frac{\lambda2^{4/3}}{\delta^{1/2}}$ and $s=\delta$ 
in Lemma \ref{lemma: local change of ptl passage time}. Thus, using the fact that the distribution of $T_N^*(s)$ does not depend on $s$ and Lemma \ref{lemma: local change of ptl passage time}, we have for each $i,$
\[
\P \left( \sup_{t_1,t_2: t_1,t_2\in I_i^\delta,}|T_n^*(t_1)-T_n^*(t_2)| \geq \lambda 2^{4/3}n^{1/3}\right) \leq Ce^{-c \frac{\lambda}{\delta^{1/2}}}.
\]
Hence, taking a union bound we get 
\[
    \P \left( \sup_{t_1,t_2: |t_1|,|t_2| \leq M, |t_1-t_2|< \delta }|T_n^*(t_1)-T_n^*(t_2)| \geq \lambda 2^{4/3}n^{1/3}\right) \leq C \frac M\delta e^{-c \frac{\lambda}{\delta^{1/2}}}.
    \]
    Now, by taking $\delta>0$ small enough we can take the right side to be smaller than $\gamma.$ This completes the proof. 
\end{proof}
So, we have to only prove the following lemma.
\begin{lemma}
\label{lemma: local change of ptl passage time}
For any $s_0>0, \varepsilon>0$ and for any $s \in [0,s_0]$ and $\ell$ sufficiently large (depending on $\vep, s_0$) and $n$ sufficiently large (depending on $s$) there exists $C,c>0$ (depending on $s_0$ and $\vep$) such that 
\[
\P \left ( \max_{|t_1|,|t_2| \leq s}|T_n^*(t_1)-T_n^*(t_2)| \geq n^{1/3}\ell^{1/3+\vep} s^{1/2} \right) \leq Ce^{-c \ell^{3 \vep/2}}.
\]
    
\end{lemma}
\begin{proof}
In this proof for convenience, we consider the point-to-line passage time instead of line-to-point passage time (i.e., define $T_n^*(t)$ to be $T_{u_0(t),\cl_{2n}}).$ Note that, by reflection invariance it is sufficient to prove the lemma for this definition.
First we note that 
\begin{align*}
\Big \{\max_{|t_1|,|t_2| \leq s}|T_n^*(t_1)-T_n^*(t_2)| \geq n^{1/3}\ell^{1/3+\vep} s^{1/2} \Big \} \subset \Big \{\max_{|t| \leq s }|T_n^*(t)-T_n^*(0)| \geq \frac{n^{1/3}\ell^{1/3+\vep} s^{1/2}}{2} \Big \} \\ \bigcup \Big\{\max_{|t| \leq s}|T_n^*(t)-T_n^*(0)| \geq \frac{n^{1/3}\ell^{1/3+\vep} s^{1/2}}{2} \Big\}.
\end{align*}
So, it is enough to prove the lemma assuming $v=0$. First we consider the case when $\ell s^{3/2} <1$ and without loss of generality assume $\ell s^{3/2} n$ is an integer. Let $I^s$ denote the interval of length $\lfloor 2sn^{2/3} \rfloor $ on $\cl_0$ with midpoint $\bo{0}$. Let $a,b$ denote the end points of $I^s$. Let us consider the line $\cl_{2 \ell s^{3/2} n}.$ Let us consider the line segment of length $\lfloor 2 \ell^{2/3+\vep} s n^{2/3} \rfloor$ on $\cl_{2 \ell s^{3/2} n}$ with midpoint $\bo{\ell s^{3/2} n}$ and we denote it by $I.$ Now, let $\mathsf{LTF}$ denote the event that for some $x \in I^s, \Gamma_{x, \cl_{2n}}$ intersects $\cl_{2 \ell s^{3/2} n}$ outside $I$. By ordering of geodesics, this will imply either $\Gamma_{a, \cl_{2n}}$ or $\Gamma_{b, \cl_{2n}}$ will intersect $\cl_{2 \ell s^{3/2} n}$ outside $I.$ By local transversal fluctuation \cite[Lemma 7.3]{BBBK24} of point-to-line geodesic, we have, there exists $c>0$ such that 
\begin{equation}
\label{eq: event LTF}
\P(\mathsf{LTF}) \leq e^{-c\ell^{3\vep}}.
\end{equation}
Now, we consider another event.\\
\begin{itemize}
    \item $\cc:= \big \{ \max_{x \in I^s, y \in I}|T_{x,y}-\E T_{x,y}| \geq n^{1/3} \ell^{1/3+\vep}s^{1/2} \big \}.$
\end{itemize}
By \cite[Theorem 4.2, (i), (ii)]{BGZ21} we have for some $c>0$ with $\ell$ sufficiently large (depending on $\vep, s_0$)
\begin{equation}
\label{eq: event C}
\P(\cc) \leq s\ell^{\vep}e^{-c \ell^{3\vep/2}} \leq e^{-c' \ell^{3 \vep/2}}.
\end{equation}
Further, we have, using \eqref{eq: expectation estimate} and a calculus argument that for any $x \in I^s$ and $y \in I,$ there exists $C>0$ (depending on $s_0$) such that for $n$ sufficiently large (depending on $s$) and $\ell$ sufficiently large (depending on $s_0$)
\[
|\E T_{x,y}-\E T_{\bo{0},y}| \leq C \ell^{1/3+\vep} n^{1/3}s^{1/2}.
\]
Now, on the event $\mathsf{LTF}^c \cap \cc^c$ we have the following. Let $x \in I^s$ and let $\Gamma_{x, \cl_{2n}}$ intersects $\cl_{2 \ell s^{3/2} n}$ at a vertex $z$. On the event $\mathsf{LTF}^c, z$ lies in $I.$ Further, on the event $\mathsf{LTF}^c \cap \cc^c,$  we have for sufficiently large $\ell$ (depending on $s_0$) and $n$ sufficiently large (depending on $s$)
\[
|T_{\bo{0},z}-T_{x,z}| \leq C \ell^{1/3+\vep} n^{1/3} s^{1/2},
\]
for some $C>0.$
Finally, we have on $\mathsf{LTF}^c \cap \cc^c,$
\[
T_{\bo{0},\cl_{2n}} \geq T_{\bo{0},z}+T_{z, \cl_{2n}} \geq T_{x,z}+T_{z, \cl_{2n}}-C \ell^{1/3+\vep} n^{1/3} s^{1/2}=T_{x, \cl_{2n}}-C \ell^{1/3+\vep} n^{1/3} s^{1/2}.
\]
For the other side, let $z$ now denote the intersection point of $\Gamma_{\bo{0}, \cl_{2n}}$ and $\cl_{2 \ell s^{3/2} n}$. Then again on $\mathsf{LTF}^c \cap \cc^c,$ we have 
\[
T_{\bo{0}, \cl_{2n}}=T_{\bo{0},z}+T_{z, \cl_{2n}} \leq T_{x,z}+C\ell^{1/3+\vep} n^{1/3}s^{1/2}+T_{z,\cl_{2n}} \leq T_{x, \cl_{2n}}+C \ell^{1/3+\vep} n^{1/3}s^{1/2}.
\]
Hence, we get 
\[
\mathsf{LTF}^c \cap \cc^c \subset \Big \{ \max_{|t|\leq s}|T_n^*(t)-T_n^*(0)| \leq C n^{1/3}  \ell^{1/3+\vep} s^{1/2}\Big \}.
\]
Hence, by \eqref{eq: event LTF} and \eqref{eq: event C} we get for some $C'>0$
\[
\P \left( \max_{|t_1|,|t_2| \leq s}|T_n^*(t_1)-T_n^*(t_2)| \geq C n^{1/3}  \ell^{1/3+\vep} s^{1/2} \right) \leq C'e^{-c \ell^{3 \vep/2}}.
\]
For the case when $\ell s^{3/2}> 1,$ we observe that
\[
\left \{\max_{|t_1|,|t_2| \leq s}|T_n^*(t_1)-T_n^*(t_2)| \geq n^{1/3}\ell^{1/3+\vep} s^{1/2} \right \} \subset \left \{ \max_{|t_1|,|t_2| \leq s}|T_n^*(t_1)-T_n^*(t_2)| \geq n^{1/3}\ell^{\vep}\right \}.
\]
Now, by the same way we argued before, we can restrict the point-to-line geodesics $\Gamma_{x, \cl_{2n}}$ starting from $I^s,$ on an interval $I$ of length $\lfloor \ell^{\vep}n^{2/3} \rfloor$ on $\cl_{2n}.$ We define an event similar to $\cc$ as 
\begin{itemize}
    \item $\cc_1:=\big \{ \max_{x \in I^s, y \in I}|T_{x,y}-\E T_{x,y}| \geq n^{1/3} \ell^{\vep} \big \}.$
\end{itemize}
Now, for $x \in I^s, y \in I$, comparing $\E \left(T_{x,y} \right)$ and $\E \left(T_{\bo{0},y} \right)$ as before and arguing similarly, we get the desired upper bound in this case. This completes the proof. 
\end{proof}

Before we finish, let us remark how Theorem \ref{t:lppairy1limit} leads to the proofs of several properties for the Airy$_1$ process.

\begin{remark}
\label{r:sm}
    Airy$_1$ process is strong mixing and hence ergodic. For $t>0$ large and $t_1,t_2, \dots t_{k}\le 0$, $t_{k+1}, t_{k+2},\ldots t_{2k}\ge t$, let $T^{**}_{N}(t_i)$ denote the weight of the best path from $\mathcal{L}_0$ to $u(t_i)$ that does not cross the line $x-y=t(2N)^{2/3}$. As in the proofs of Theorems \ref{thm: point to line max limit lower bound} and \ref{thm: point to line min limit lower bound},  $\{T^{**}_{N}(t_i)\}_{i\le k}$ are independent of $\{T^{**}_{N}(t_i)\}_{i>k}$ and $\P(\exists i : T^*_{N}(t_i)\neq T^{**}_{N}(t_i))\le e^{-ct^3}$. It is easy to see that this, together with Theorem \ref{t:lppairy1limit} implies that $\ca_1$ is strong mixing. The observation that $\ca_1$ is strong mixing appears to be new, but ergodicity was proved in \cite[Theorem 1.1]{P23} using the association property of $\ca_1$. Strong mixing of $\ca_2$ is also known \cite{PSpng02}. One would expect that strong mixing of $\ca_2$ could also be proved using LPP arguments, but it is not immediate from our estimates. 
\end{remark}

\begin{remark}
    \label{r:ass}
    The processes $\ca_1$ and $\ca_2$ are associated. Indeed one can see that $T_{N}(s_i)$ and $T^*_{N}(s_i)$ are increasing functions of independent random variables therefore any finite collection of them are associated \cite{EPW67}. The weak convergence results Theorems \ref{t:lppairy1limit} and \ref{t:lppairy2limit} then implies that the processes $\ca_1(\cdot)$ and $\ca_2(\cdot)$ are also associated using \cite{EPW67}. This result already appears in \cite[Theorem 1.2]{P23}, where it is proved using a result from \cite{CKNP23} about association property of the stochastic heat equation (SHE) starting from the narrow wedge and flat initial conditions. This result is crucially used in \cite{P23} in several other proofs. 
\end{remark}

\bibliography{intersection}
\bibliographystyle{plain}
\end{document}